\documentclass[reqno,12pt]{amsart}

\usepackage{mathrsfs}
\usepackage{amsmath,amssymb,slashed}
\usepackage{color,xcolor}
\usepackage[english]{babel}
\usepackage[utf8]{inputenc}
\usepackage{amsmath,amscd}
\usepackage{amsfonts}
\usepackage{amsthm}
\usepackage{graphicx}
\usepackage[colorinlistoftodos]{todonotes}
\usepackage{amsmath,amssymb,amsfonts}
\usepackage{underscore}
\usepackage{xy}
\usepackage[T1]{fontenc}
\usepackage[utf8]{inputenc}
\usepackage{amsthm}
\usepackage{mathrsfs}
\usepackage{amsmath,amssymb,slashed}
\usepackage{color,xcolor}

 \usepackage{fullpage}
\input xy
\xyoption{all}

\theoremstyle{plain}
\newtheorem{thm}{Theorem}[section]
\newtheorem{prop}[thm]{Proposition}
\newtheorem{lem}[thm]{Lemma}
\newtheorem{cor}[thm]{Corollary}

\newtheorem{pro}[thm]{Problem}

\theoremstyle{definition}
\newtheorem{defn}[thm]{Definition}

\theoremstyle{remark}
\newtheorem{rmk}[thm]{Remark}

\begin{document}

\title[On the Takai duality for $L^{p}$ operator crossed products]{On the Takai duality for $L^{p}$ operator crossed products}

\author[Z. Wang]{Zhen Wang}
\curraddr{Department of Mathematics\\Jilin University\\Changchun 130012\\P.~R. China}
\address{Key Laboratory of Applied Mathematics of Fujian Province University\\School of Mathematics and Finance \\Putian University \\Putian 351100\\
P. R. China}
\email{zwangmath@jlu.edu.cn}

\author[S. Zhu]{Sen Zhu}
\address{Department of Mathematics\\Jilin University\\Changchun 130012\\P.~R. China}
\email{zhusen@jlu.edu.cn}


\begin{abstract}
The aim of this paper is to study a problem raised by N. C. Phillips concerning the existence of Takai duality for $L^p$ operator crossed products $F^{p}(G,A,\alpha)$, where $G$ is a locally compact Abelian group, $A$ is an $L^{p}$ operator algebra and $\alpha$ is an isometric action of $G$ on $A$. Inspired by D. Williams' proof for the Takai duality theorem for crossed products of $C^*$-algebras, we construct a homomorphism  $\Phi$ from $F^{p}(\hat{G},F^p(G,A,\alpha),\hat{\alpha})$ to $\mathcal{K}(l^{p}(G))\otimes_{p}A$ which is a natural $L^p$-analog of D. Williams' map. For countable discrete Abelian groups $G$ and separable unital $L^p$ operator algebras $A$ which have unique $L^p$ operator matrix norms,
we show that $\Phi$ is an isomorphism if and only if either $G$ is finite or $p=2$; in particular, $\Phi$ is an isometric isomorphism in the case that $p=2$. Moreover, it is proved that $\Phi$ is equivariant for the double dual action $\hat{\hat{\alpha}}$ of $G$ on $F^p(\hat{G},F^p(G,A,\alpha),\hat{\alpha})$ and the action $\mathrm{Ad}\rho\otimes\alpha$ of $G$ on $\mathcal{K}(l^p(G))\otimes_p A$.
\end{abstract}

\date{\today}
\maketitle


\section{Introduction}
The Takai duality theorem is an important result concerning crossed products of $C^{*}$-algebras.
Recall that a $C^{*}$-dynamical system is a triple $(A,G,\alpha)$ consisting of a $C^{*}$-algebra $A$, a locally compact group $G$ and a continuous homomorphism $\alpha:G\rightarrow \mathrm{Aut}(A)$, where ${\rm Aut}(A)$ is the group of automorphisms of $A$.
Let $A\rtimes_\alpha G$ be the full crossed product of $A$ by $G$. If $G$ is an Abelian group with dual group $\hat{G}$, then there is a dual action $\hat{\alpha}$ of $\hat{G}$ on $A\rtimes_\alpha G$ which is given by $\hat{\alpha}_\gamma:C_c(G,A)\rightarrow C_c(G,A),~\hat{\alpha}_\gamma(f)(s):=\overline{\gamma(s)}f(s)$ for all $\gamma\in \hat{G}, s\in G$ and $f\in C_c(G,A)$, and $\hat{\alpha}_\gamma$ can extend to $A\rtimes_\alpha G$ by continuity. Therefore $(A\rtimes_{\alpha}G,\hat{G},\hat{\alpha})$ is a $C^*$-dynamical system called the
dual system.
The Takai duality theorem \cite{Takai} says that the iterated crossed product $(A\rtimes_{\alpha}G)\rtimes _{\hat{\alpha}}\hat{G}$ is $*$-isomorphic to $A\otimes \mathcal{K}(L^{2}(G))$,
where $\mathcal{K}(L^{2}(G))$ denotes the algebra of compact operators on $L^{2}(G)$.
Moreover, this isomorphism is also equivariant for the double dual action $\hat{\hat{\alpha}}$ of $G$ on $(A\rtimes_{\alpha}G)\rtimes _{\hat{\alpha}}\hat{G}$ and the action $\alpha\otimes \mathrm{Ad}\rho$ of $G$ on $A\otimes \mathcal{K}(L^{2}(G))$.
This result has many applications in the $K$-theory of operator algebras. For example, it can be used to prove Connes' Thom isomorphism theorem (see \cite[Theorem 10.2.2]{Blackadar} or \cite{Connes}). In 1986, I. Raeburn \cite{Raeburn} gave another proof of Takai's theorem by exploiting the universal properties of crossed products of $C^*$-algebras. D. Williams provided a variation of I. Raeburn's proof, with an extra contribution
by S. Echterhoff (see \cite[Theorem 7.1]{Williams}).



%


The aim of this paper is to study a problem raised by N. C. Phillips \cite[Problem 8.7]{PlpsOpenQues} concerning
the Takai duality for $L^p$ operator crossed products. To proceed, we first introduce some notations and terminology. The reader is referred to \cite{Gardella and Thiel convolution} or \cite{N. C. Phillips Lp} for more details.

The class of $L^p$ operator algebras is a natural generalization of operator algebras on Hilbert spaces,
and was first studied in Herz's influential paper \cite{Herz} on harmonic analysis on $L^p$ spaces.
People's renewed interests in $L^p$ operator algebras were inspired by N. C. Phillips
\cite{N. C. Phillips Lp}, who showed that a few of the important constructions of $C^*$-algebras can be
generalized to the setting of $L^p$ operator algebras.
Recently the study of $L^p$ operator algebras has received much attention, and
a number of authors have made significant contributions to it (see \cite{Blecher and Phillips,Choi and Gardela,Gardella and Thiel groupoid,Gardella and Thiel,Gardella and Thiel quotient,Gardella and Thiel convolution operators,Gardella and Thiel convolution,Gardella modern,Hejazian and Pooya,N. C. Phillips Lp}).

The $L^p$ operator crossed products were introduced by N. C. Phillips \cite{N. C. Phillips Lp} with the aim to compute the $K$-theory groups of $L^{p}$ Cuntz algebras. Recall that a Banach algebra $A$ is an $L^{p}$ {\it operator algebra} if it can be isometrically represented on an $L^{p}$ space $(p\in [1, \infty))$.
Let $\pi:A\rightarrow \mathcal{B}(L^p(X,\nu))$ be a representation of $A$. We say that $\pi$ is {\it $\sigma$-finite} if $\nu$ is $\sigma$-finite.
Let $G$ be a locally compact group. Then there exists a left Haar measure $\mu$ on $G$. An $L^{p}$ {\it operator algebra dynamical system} is a triple $(G,A,\alpha)$ consisting of a locally compact group $G$, an $L^{p}$ operator algebra $A$ and a continuous homomorphism $\alpha:G\rightarrow \mathrm{Aut}(A)$, where $\mathrm{Aut}(A)$ is the group of isometric automorphisms of $A$.
A {\it contractive covariant representation} of $(G,A,\alpha)$ on an $L^{p}$ space $E$ is a pair $(\pi,v)$ consisting of a nondegenerate contractive homomorphism $\pi:A\rightarrow \mathcal{B}(E)$ and an isometric group representation $v:G\rightarrow \mathcal{B}(E)$, satisfying the covariance condition $$v_{t}\pi(a)v_{t^{-1}}=\pi(\alpha_{t}(a))$$ for all $t\in G$ and $a\in A$.
A covariant representation $(\pi,v)$ of $(G,A,\alpha)$ is {\it $\sigma$-finite} if $\pi$ is $\sigma$-finite.

Denote by $C_{c}(G,A,\alpha)$ the vector space of continuous compactly supported functions $G\rightarrow A$, made into an algebra over $\mathbb{C}$ with product given by twisted convolution, that is, $$(f*g)(t):=\int_{G}f(s)\alpha_{s}(g(s^{-1}t))d\mu(s)$$ for all $f,g\in C_{c}(G,A,\alpha)$ and $t\in G$. The {\it integrated form} of $(\pi,v)$ is the nondegenerate contractive homomorphism $\pi\rtimes v: C_{c}(G,A,\alpha)\rightarrow \mathcal{B}(E)$ given by $$(\pi\rtimes v)(f)(\xi):=\int_{G}\pi(f(t))v_{t}(\xi)d\mu(t)$$ for all $f\in C_{c}(G,A,\alpha)$ and for all $\xi\in E$. Denote by $\mathrm{Rep}_{p}(G,A,\alpha)$ the class of all nondegenerate $\sigma$-finite contractive covariant representations of $(G,A,\alpha)$ on $L^{p}$-spaces. The {\it full $L^p$ operator crossed product} $F^{p}(G,A,\alpha)$ is defined as the completion of $C_{c}(G,A,\alpha)$ in the norm $$||f||_{F^{p}(G,A,\alpha)}:=\sup\big\{||(\pi\rtimes v)(f)||:(\pi,v)\in \mathrm{Rep}_{p}(G,A,\alpha)\big\}.$$

Let $(G,A,\alpha)$ be an $L^{p}$ operator algebra dynamical system.
Given a nondegenerate $\sigma$-finite contractive representation $\pi_{0}:A\rightarrow \mathcal{B}(E_{0})$ on an $L^{p}$ space $E_{0}$, its associated {\it regular covariant representation} is the pair $(\pi,\lambda_{p}^{E_{0}})$ on $L^{p}(G)\otimes_{p} E_{0}\cong L^{p}(G,E_{0})$ given by
$$\pi(a)(\xi)(s):=\pi_{0}\left(\alpha_{s^{-1}}(a)\right)(\xi(s)) \quad$$ and $$\lambda_p^{E_0}(s)(\xi)(t):=\xi(s^{-1}t)$$
for all $a\in A$, $\xi\in L^{p}(G,E_{0})$, and $s,t\in G$.
We denote by $\mathrm{RegRep}_{p}(G,A,\alpha)$ the class consisting
of nondegenerate $\sigma$-finite contractive regular covariant representations of $(G,A,\alpha)$, which is clearly a subclass of $\mathrm{Rep}_{p}(G,A,\alpha)$. The {\it reduced $L^p$ operator crossed product} $F^{p}_{\lambda}(G,A,\alpha)$ is defined as the completion of $C_{c}(G,A,\alpha)$ in the norm $$||f||_{F^{p}_{\lambda}(G,A,\alpha)}:=\sup\{||(\pi\rtimes v)(f)||:(\pi,v)\in \mathrm{RegRep}_{p}(G,A,\alpha)\}.$$

By \cite[Theorem 7.1]{Phillips look like}, if $G$ is amenable, then the identity map on $C_c(G,A,\alpha)$ can be extended to an isometric isomorphism from $F^{p}(G,A,\alpha)$ onto $ F^{p}_{\lambda}(G,A,\alpha)$.
If $A=\mathbb{C}$, then it is easy to see that $F^{p}(G,A,\mathrm{id})$ is the full group $L^{p}$ operator algebra $F^{p}(G)$ and $F^{p}_{\lambda}(G,A,\mathrm{id})$ is the reduced group $L^{p}$ operator algebra $F^{p}_{\lambda}(G)$,
where $\mathrm{id}$ is the trivial action of $G$ on $\mathbb{C}$.

If, in addition, $G$ is Abelian, then we let $\hat{G}$ denote the dual group of $G$. For each $\gamma\in \hat{G}$, N. C. Phillips defined an isomorphism $\hat{\alpha}_{\gamma}:C_{c}(G,A,\alpha)\rightarrow C_{c}(G,A,\alpha)$ by $\hat{\alpha}_{\gamma}(f)(s):=\overline{\gamma(s)}f(s)$ for all $f\in C_{c}(G,A,\alpha)$ and $s\in G$ (see \cite[Definition 3.15]{N. C. Phillips Lp}). Thus $\hat{\alpha}_{\gamma}$ can extend to an isometry on $F^{p}(G,A,\alpha)$ by continuity (see \cite[Theorem 3.18]{N. C. Phillips Lp}).
Hence there is a {\it dual system} $(\hat{G},F^{p}(G,A,\alpha),\hat{\alpha})$.
So we obtain the iterated $L^{p}$ operator crossed product $F^{p}(\hat{G},F^{p}(G,A,\alpha),\hat{\alpha})$.


In {\cite[Problem 8.7]{PlpsOpenQues}}, N. C. Phillips raised the following natural problem concerning the existence of the Takai duality theorem for $L^{p}$ operator crossed products.

\begin{pro}\label{P:main}
Let $p\in[1,\infty)$, $A$ be an $L^{p}$ operator algebra, $G$ be a locally compact Abelian group, and $\alpha:G\rightarrow \mathrm{Aut}(A)$ be an isometric action of $G$ on $A$. Then is $F^{p}(\hat{G},F^{p}(G,A,\alpha),\hat{\alpha})$ isomorphic to $\mathcal{K}(L^{p}(G))\otimes_{p}A$?
Here $\mathcal{K}(L^{p}(G))$ denotes the algebra of compact operators on $L^{p}(G)$ and $\otimes_{p}$ denotes the spatial $L^{p}$ operator tensor product of $L^{p}$ operator algebras.
\end{pro}

The reader is referred to \cite{N. C. Phillips Lp} for the details of spatial $L^{p}$ operator tensor products.

The aim of this paper is to consider Problem \ref{P:main} following D. Williams' proof for the Takai duality theorem. We first briefly recall D. Williams' approach (see \cite[Theorem 7.1]{Williams}). Suppose that $G$ is a locally compact Abelian group and $(A,G,\alpha)$ is a $C^*$-dynamical system. D. Williams constructed a map $\Psi$ from $(A\rtimes_\alpha G)\rtimes_{\hat{\alpha}} \hat{G}$ to $ \mathcal{K}(L^2(G))\otimes A$ as a composition of four $*$-homomorphisms

\begin{align}\label{Eq:Williams}
\begin{CD}
(A\rtimes_\alpha G)\rtimes_{\hat{\alpha}} \hat{G} @>\Psi_1>> (A\rtimes_{\mathrm{id}} \hat{G})\rtimes_{\hat{\mathrm{id}}^{-1}\otimes \alpha} G @ >\Psi_2>>C_0(G,A)\rtimes_{\mathrm{lt}\otimes\alpha} G \\
@.     @. @VV\Psi_3V\\
@.     \mathcal{K}(L^2(G))\otimes A @<\Psi_4<< C_0(G,A)\rtimes_{\mathrm{lt}\otimes\mathrm{id}} G
\end{CD},
\end{align}
and proved that $\Psi_1, \Psi_2, \Psi_3$ and $\Psi_4$ are isomorphisms. It follows readily
that $\Psi=\Psi_4\circ\Psi_3\circ\Psi_2\circ\Psi_1$ is an isomorphism.

Our approach to the Takai duality for $L^{p}$ operator crossed products is inspired by the preceding construction.
Unless otherwise stated, we assume throughout the following that $(G,A,\alpha)$ is an $L^p$ operator algebra dynamical system,
where $G$ is a countable discrete Abelian group, and $A$ is a separable unital $L^p$ operator algebra which has unique $L^p$ operator matrix norms (see Definition \ref{matrix norm}). We construct $L^p$-analogues of the construction (\ref{Eq:Williams}) as follows:
$$\begin{CD}
F^{p}(\hat{G},F^{p}(G,A,\alpha),\hat{\alpha}) @>\Phi_1>> F^{p}(G,
F^{p}(\hat{G},A,\beta),\hat{\beta}\otimes \alpha) @ >\Phi_2>>F^{p}(G,C_{0}(G,A),\mathrm{lt}\otimes\alpha) \\
@.     @. @VV\Phi_3V\\
@.      \mathcal{K}(l^{p}(G))\otimes_{p}A @<\Phi_4<< F^{p}(G,C_{0}(G,A),\mathrm{lt}\otimes\mathrm{id})
\end{CD} .$$
Next we will explain the constructions of $\Phi_{1}, \Phi_{2}, \Phi_{3}$ and $\Phi_{4}$.


{\it The construction of $\Phi_1$}.
Let $\beta$ be the trivial action of $\hat{G}$ on $A$. For each $s\in G$, we define an isometric isomorphism $ \hat{\beta}_{s}:C(\hat{G},A,\beta)\rightarrow C(\hat{G},A,\beta)$ by $\hat{\beta}_{s}(\varphi)(\gamma):=\gamma(s)\varphi(\gamma)$ for all $\varphi\in C(\hat{G},A,\beta)$ and $\gamma\in \hat{G}$. Then we can form the tensor product action of $G$ on $F^{p}(\hat{G},A,\beta)$ by
 $(\hat{\beta}\otimes\alpha)_{s}(\varphi)(\gamma):=\gamma(s)\alpha_{s}(\varphi(\gamma)).$ Thus $(\hat{\beta}\otimes \alpha)_{s}$ extends to an element of $\mathrm{Aut}(F^{p}(\hat{G},A,\beta))$.
For each $F\in C_{c}(\hat{G}\times G,A)$, we define $\Phi_1(F)$ as
\begin{align}\label{Phi1}
\Phi_{1}(F)(s,\gamma):=\gamma(s)F(\gamma,s), \ \ \  s\in G, \gamma\in \hat{G}.
\end{align}
It will be shown that $\Phi_1$ is an isometric homomorphism with dense range (see Proposition \ref{change}). Then $\Phi_1$ extends to an isometric isomorphism from $F^{p}(\hat{G},F^{p}(G,A,\alpha),\hat{\alpha})$ onto $F^{p}(G,
F^{p}(\hat{G},A,\beta),\hat{\beta}\otimes \alpha)$.

{\it The construction of $\Phi_2$}.
Let $\mathrm{lt}:G\rightarrow \mathrm{Aut}(C_{0}(G))$ be the isometric
action induced by left translation of $G$ on itself. Then we can form the tensor product action of $G$ on $C_{0}(G,A)$ by $(\mathrm{lt}\otimes \alpha)_{t}f(s):=\alpha_{t}(f(t^{-1}s))$ for all $f\in C_{0}(G,A)$ and $s, t\in G$.
Let $\Gamma_p:F^p_\lambda(\hat{G})\rightarrow C_0(G)$ is the Gelfand transformation of $F^p_\lambda(\hat{G})$.
Since the Gelfand transformation coincides with the Fourier transformation  on $L^1(\hat{G})$, we have $\Gamma_p(\varphi)(t)=\hat{\varphi}(t)=\int_{\hat{G}} \varphi(\gamma)\overline{\gamma(t)}d\mu(\gamma)$ for all $\varphi\in L^1(\hat{G})$.
For $\varphi\in C(\hat{G})$ and $a\in A$, the linear span of the elementary tensors $\varphi\otimes a$ are dense in $F^{p}(\hat{G},A,\beta)$. Since $\beta$ is the trivial action of $\hat{G}$ on $A$,
by the universal property of $F^{p}(\hat{G},A,\beta)$ (see \cite[Theorem 3.6]{N. C. Phillips Lp}), there exists a contractive homomorphism from $F^{p}(\hat{G},A,\beta)$ to $C_{0}(G,A)$ by sending  $\varphi\otimes a$ to $\Gamma_{p}(\varphi)\otimes a$. We denote this homomorphism by $
\varphi_2$.
It will be shown that $\varphi_2$ is equivariant for the action $\hat{\beta}\otimes\alpha$ of $G$ on $F^p(\hat{G},A,\beta)$ and the action $\mathrm{lt}\otimes\alpha$ of $G$ on $C_0(G,A)$. By the universal property of $F^{p}(G,F^{p}(\hat{G},A,\beta),\hat{\beta})$ (see \cite[Theorem 3.6]{N. C. Phillips Lp}),
 there is a homomorphism $\Phi_{2}:=\varphi_2\rtimes\mathrm{id}:F^{p}(G,F^{p}(\hat{G},A,\beta),\hat{\beta}\otimes\alpha)\rightarrow F^{p}(G,C_{0}(G,A),\mathrm{lt}\otimes\alpha)$ (see Proposition \ref{injective}) which maps $F\in C_{c}(G\times \hat{G},A)$ into $C_c(G,C_0(G,A))$ by the formula
\begin{align}\label{Phi2}
\Phi_{2}(F)(s,t):=\int_{\hat{G}}F(s,\gamma)\overline{\gamma(t)}d\mu(\gamma), \ \ \ \ \ s,t\in G.
\end{align}
We remark that $\Phi_2$ is essentially induced by the Gelfand transformation $\Gamma_p$ of $F^p_\lambda(\hat{G})$.

{\it The construction of $\Phi_3$}.
Let $\mathrm{id}$ be the trivial action of $G$ on $A$. Then we can form the tensor product action of $G$ on $C_{0}(G,A)$ by $(\mathrm{lt}\otimes \mathrm{id})_{t} f(s):=f(t^{-1}s)$ for all $s, t\in G$ and $f\in C_0(G,A)$.
For $F\in C_{c}(G,C_{0}(G,A))$, let $\Phi_3(F)$ be defined as
\begin{align}\label{Phi3}
\Phi_{3}(F)(s,t):=\alpha^{-1}_{t}(F(s,t)),\ \ \ \ s,t\in G.
\end{align}
It will be also shown that $\Phi_3$ is isometric (see Proposition \ref{change for trival action}). Hence $\Phi_3$ extends to an isometric isomorphism from $F^{p}(G,C_{0}(G,A),\mathrm{lt}\otimes\alpha)$ onto $F^{p}(G,C_{0}(G,A),\mathrm{lt}\otimes\mathrm{id})$.

{\it The construction of $\Phi_4$}.
E. Gardella and H. Thiel {\cite[Theorem 4.3]{Gardella and Thiel}}
have showed that $F^{p}(G,C_{0}(G),\mathrm{lt})$ is isometrically isomorphic to $\overline{M}_{G}^{p}$
(see \cite[Example 1.6]{N. C. Phillips Lp} for the definition of $\overline{M}_{G}^{p}$).
$\overline{M}_{G}^{p}$ equals $\mathcal{K}\left(l^{p}(G)\right)$ when $p>1$, and is strictly smaller than $\mathcal{K}\left(l^{1}(G)\right)$ when $p=1$ and $G$ is infinite.  Since $A$ has unique $L^p$ operator matrix norms, the norm on spatial $L^p$ operator tensor product $\mathcal{K}\left(l^{p}(G)\right)\otimes_{p}A$ is independent of the choice of isometric representation of $A$ (see Lemma \ref{isometric tensor}).
We assume that $A\subseteq \mathcal{B}(E)$ is a closed subalgebra on some separable $L^p$ space $E$.
By the universal property of full $L^p$ operator crossed product $F^{p}(G,C_{0}(G,A),\mathrm{lt}\otimes \mathrm{id})$ (see \cite[Theorem 3.6]{N. C. Phillips Lp}), there exists a homomorphism
$\Phi_{4}:F^{p}(G,C_{0}(G,A),\mathrm{lt}\otimes\mathrm{id})\rightarrow \overline{M}_{G}^{p}\otimes_{p}A\subseteq \mathcal{K}\left(l^{p}(G)\right)\otimes_{p}A$ which maps $F\in C_{c}(G\times G,A)\subseteq F^{p}(G,C_{0}(G,A),\mathrm{lt}\otimes \mathrm{id})$ to $\mathcal{K}\left(l^{p}(G)\right)\otimes_{p}A$
by the formula \begin{align}\label{Phi4}
\Phi_{4}(F)h(t):=\sum_{s\in G}F(s,t)h(s^{-1}t),
\end{align}
where $h\in C_{c}(G,E)\subseteq l^{p}(G,E)$.


We will prove that the map $\Phi=\Phi_4\circ\Phi_3\circ\Phi_2\circ\Phi_1$ is an equivariant homomorphism for the
double dual action $\hat{\hat{\alpha}}$ of $G$ on $F^{p}(\hat{G},F^{p}(G,A,\alpha),\hat{\alpha})$ and the action $\mathrm{Ad}\rho\otimes \alpha$ of $G$ on $\mathcal{K}(l^{p}(G))\otimes_{p}A$.
The Pontryagin Duality Theorem allows us to identify $G$ with the dual of $\hat{G}$ where $s\in G$ is associated to the character $\gamma\mapsto \gamma(s)$ on $\hat{G}$. Therefore there is a
{\it double dual action} $\hat{\hat{\alpha}}$ of $G$ on the iterated crossed product $F^{p}(\hat{G},F^{p}(G,A,\alpha),\hat{\alpha})$. The double dual action $\hat{\hat{\alpha}}$ of $G$ on $C_{c}(\hat{G}\times G,A)\subseteq F^{p}(\hat{G},F^{p}(G,A,\alpha),\hat{\alpha})$ is given by $\hat{\hat{\alpha}}_{t}(F)(\gamma, s):=\overline{\gamma(t)} F(\gamma, s)$ for all $t\in G$. Next let us introduce the action of $G$ on $\mathcal{K}(l^{p}(G))\otimes_{p}A$.
Let $\rho:G\rightarrow \mathcal{B}(l^{p}(G))$ be the right-regular representation and form the dynamical system $\mathrm{Ad}\rho: G\rightarrow \mathrm{Aut}(\mathcal{K}(l^{p}(G)))$, where $(\mathrm{Ad}\rho)_{s}T:=\rho(s)T\rho(s^{-1})$. Then we have the tensor product dynamical system $\mathrm{Ad}\rho\otimes\alpha:G\rightarrow \mathrm{Aut}(\mathcal{K}(l^{p}(G)))\otimes_{p}A$, where $(\mathrm{Ad}\rho\otimes\alpha)_s (T\otimes a)=(\mathrm{Ad}\rho)_s(T)\otimes \alpha_s(a)$.


The main result of this paper is the following theorem.

\begin{thm}\label{T:main}
Let $(G,A,\alpha)$ be an $L^p$ operator algebra dynamical system,
where $G$ is a countable discrete Abelian group, and $A$ is a separable unital $L^p$ operator algebra which has unique $L^p$ operator matrix norms.
Let $\Phi_1$, $\Phi_2$, $\Phi_3$, $\Phi_4$ be defined as in (\ref{Phi1}), (\ref{Phi2}), (\ref{Phi3}),   (\ref{Phi4}) and $\Phi=\Phi_4\circ\Phi_3\circ\Phi_2\circ\Phi_1$.
 Then
\begin{enumerate}
\item[(i)] $\Phi_{1}$ and $\Phi_{3}$  are two isometric isomorphisms for $p\in [1,\infty)$;
\item[(ii)] $\Phi_2$ is an isomorphism if and only if either $G$ is finite or $p=2$; in particular, $F^2(\hat{G},F^2(G,A,\alpha),\hat{\alpha})$ is isometrically isomorphic to $F^{p}(G,C_{0}(G,A),\mathrm{lt}\otimes\alpha)$;
\item[(iii)]  $\Phi_4$ is an isometric homomorphism for $p\in [1,\infty)$, and $\Phi_4$ is an isometric isomorphism if and only if $p\in (1,\infty)$;
\item[(iv)] $\Phi$ is equivariant for the double dual action $\hat{\hat{\alpha}}$ of $G$ on $F^p(\hat{G},F^p(G,A,\alpha),\hat{\alpha})$ and the action $\mathrm{Ad}\rho\otimes\alpha$ of $G$ on $\mathcal{K}(l^p(G))\otimes_p A$.
\end{enumerate}
\end{thm}

\begin{cor}\label{Corollary}
Under the same assumption in Theorem \ref{T:main}, the map
$\Phi:F^p(\hat{G},F^p(G,A,\alpha),\hat{\alpha})$  $\rightarrow \mathcal{K}(l^p(G))\otimes_p A$ is an isomorphism if and only if either $G$ is finite or $p=2$. In particular, $F^2(\hat{G},F^2(G,A,\alpha),\hat{\alpha})$ is isometrically isomorphic to $\mathcal{K}(l^2(G))\otimes_2 A$.
\end{cor}

\begin{rmk}
By Theorem \ref{T:main}, Problem \ref{P:main} is reduced to studying when $F^{p}(G,
F^{p}(\hat{G},A,\beta),\hat{\beta}\otimes \alpha)$ is isomorphic to $F^{p}(G,C_{0}(G,A),\mathrm{lt}\otimes\alpha)$.
Our result shows that the natural homomorphism $\Phi_2$ between them, induced by the Gelfand transformation of $F^p_\lambda(\hat{G})$,
is not an isomorphism. This suggests a negative answer to N. C. Phillips' problem on the $L^p$ Takai duality in the case that $p\in [1,\infty)\setminus\{2\}$ and $G$ is infinite.
\end{rmk}

\begin{rmk}
Usually the norm on spatial $L^{p}$ operator tensor product depends on how an $L^p$ operator algebra is represented (see \cite[page 8]{N. C. Phillips Lp}). This differs significantly from the $C^*$-algebra case.
To ensure the ``uniqueness" of the norm on the spatial $L^p$ operator tensor product, we only consider
$L^p$ operator algebras with unique $L^p$ operator matrix norms. Furthermore, the integrated form of a nondegenerate $\sigma$-finite isometric regular covariant representation of $(G,A,\alpha)$ is an isometric representation of $F^p_\lambda(G,A,\alpha)$ (see Lemma \ref{does not depend}).
We do not know whether the result of Theorem \ref{T:main} still holds without the assumption that $A$ has unique $L^p$ operator matrix norms.

\end{rmk}


The paper is organized as follows.
In Section 2, we make some preparation for the proof of Theorem \ref{T:main}.
Subsection 2.1 is devoted to the description of nondegenerate contractive representation of $F^{p}(G,A,\alpha)$, showing that each such representation is the integrated form for some nondegenerate contractive covariant representation of $(G,A,\alpha)$ (see Theorem \ref{correspondence}), which will be useful in the proof of Theorem \ref{T:main} (i). In Subsection 2.2, we prove two properties of $L^p$ operator algebras with unique $L^p$ operator matrix norms (see Lemmas \ref{isometric tensor} and \ref{does not depend}). Section 3 is devoted to the proof of Theorem \ref{T:main}.

\section{Preparation}

In this section we make some preparation for the proof of Theorem \ref{T:main}.
\subsection{Contractive representations of $F^{p}(G,A,\alpha)$}
$\\$
$\\$
This subsection is devoted to the description of nondegenerate contractive representation of $F^{p}(G,A,\alpha)$.
The following theorem is the main result, which will be useful to prove Theorem \ref{T:main} (i).


\begin{thm}\label{correspondence}
Let $(G,A,\alpha)$ be an $L^p$ operator algebra dynamical system, where $G$ is a locally compact group, and $A$ is an $L^{p}$ operator algebra with a two-sided contractive approximate identity.
Then every nondegenerate $\sigma$-finite contractive representation of $F^{p}(G,A,\alpha)$ has the
form $\pi\rtimes v$ for some nondegenerate $\sigma$-finite contractive covariant representation $(\pi,v)$ of $(G,A,\alpha)$.
\end{thm}

\begin{proof} We let $L:F^{p}(G,A,\alpha)\rightarrow \mathcal{B}(E)$ be a nondegenerate $\sigma$-finite contractive representation of $F^{p}(G,A,\alpha)$ on an $L^{p}$ space $E$. Next we shall construct a nondegenerate $\sigma$-finite contractive covariant representation of $(G,A,\alpha)$ on $E$ admitting $L$ as its integrated form. The full $L^{p}$ operator crossed product $F^{p}(G,A,\alpha)$ in general does not contain a copy of either $A$ or $G$. However, the left multiplier algebra of $F^{p}(G,A,\alpha)$ does.
Our main approach is to construct maps from $A$ and $G$ into the left multiplier algebra of $F^{p}(G,A,\alpha)$.
Let $M_{L}(F^{p}(G,A,\alpha))$ be the left multiplier algebra of $F^{p}(G,A,\alpha)$, and we write $\iota:F^{p}(G,A,\alpha)\rightarrow M_{L}(F^{p}(G,A,\alpha))$ for the canonical inclusion. Then, by \cite[Theorem 4.5]{Gardella and Thiel convolution}, there exists a canonical extension $\overline{L}:  M_{L}(F^{p}(G,A,\alpha))\rightarrow \mathcal{B}(E)$ such that $L=\overline{L}\circ \iota.$

{\bf Step 1.} The construction of an injective homomorphism $i_{A}:A\rightarrow M_{L}(F^{p}(G,A,\alpha))$.

For each $f\in C_{c}(G,A,\alpha)$, we define $i_{A}(a)f(s):=af(s)$.
Clearly, if $(\pi_1,u)$ is a contractive covariant representation of $(G,A,\alpha)$ on an $L^{p}$ space $E_{0}$, then we have
\begin{align}\label{Eq:a}
\pi_1 \rtimes u(i_{A}(a)f)=\pi_1(a)\circ \pi_1\rtimes u (f).
\end{align}
It follows that $$||i_{A}(a)f||\leq ||a||||f||.$$ Therefore $i_{A}(a)$ extends to a map from $F^{p}(G,A,\alpha)$ to itself.
For $f,g\in C_{c}(G,A,\alpha)$, note that $$\begin{aligned} (i_{A}(a)f)*g(t) &=
 \int_{G} i_{A}(a)f(s)\alpha_{s}(g(s^{-1}t))d\mu(s)
  \\ &=a \int_{G}f(s)\alpha_{s}(g(s^{-1}t))d\mu(s)
  \\ &=i_{A}(a)(f*g)(t).\end{aligned}$$
Since $i_{A}(a)$ is bounded and $C_c(G,A,\alpha)$ is dense in $F^p(G,A,\alpha)$,
it follows that $i_{A}(a)\in M_{L}(F^{p}(G,A,\alpha))$.
Obviously, $i_{A}$ is a homomorphism from $A$ to $M_{L}(F^{p}(G,A,\alpha))$.

To see that $i_A$ is injective, let $\rho_0: A\rightarrow \mathcal{B}(E_1)$ be an faithful nondegenerate  representation of $A$ on an $L^p$ space $E_1$, and let $(\rho, \lambda_p^{E_1})$ be its associated
regular covariant representation on $L^p(G,E_1)$. Then $\rho: A\rightarrow \mathcal{B}(L^p(G,E_1))$ is faithful and nondegenerate (see \cite[Lemma 2.11(7)]{N. C. Phillips Lp}. Let $\overline{\rho\rtimes \lambda_p^{E_1}}:M_L(F^p(G,A,\alpha))\rightarrow \mathcal{B}(L^p(G,E_1))$ be the canonical extension of $\rho\rtimes \lambda_p^{E_1}$. Since $A$ has a two-sided contractive approximate identity, by \cite[Theorem 5.5]{Gardella and Thiel convolution},
there is a two-sided contractive approximate identity in $F^{p}(G,A,\alpha)$. Since $(\rho, \lambda_p^{E_1})$ is nondegenerate, by (\ref{Eq:a}), we have
 $\rho(a)=\overline{\rho\rtimes \lambda_p^{E_1}}\circ i_A(a)$. It follows that $i_A$ is injective.

{\bf Step 2.} The construction of an injective homomorphism from $G$ into the group of invertible isometries in $M_{L}(F^{p}(A,G,\alpha))$.

For each $s\in G$, we define $i_{G}(s):C_{c}(G,A,\alpha)\rightarrow C_{c}(G,A,\alpha)$ as $i_{G}(s)f(t):=\alpha_{s}(f(s^{-1}t))$ for   $f\in C_{c}(G,A,\alpha)$ and $t\in G$.  Note that if $(\pi_2,w)$ is a contractive covariant representation of $(G,A,\alpha)$ on an $L^{p}$-space $E_{2}$, then for $\xi\in E_{2}$ we have
$$\begin{aligned}\label{eq2} \pi_2\rtimes w(i_{G}(s)f))(\xi) &=\int_{G} \pi_2(i_{G}(s)f(t))w_{t}(\xi)d\mu(t) \\ &=\int_{G} \pi_2(\alpha_{s}(f(s^{-1}t)))w_{t}(\xi)d\mu(t)\\ &=\int_{G} \pi_2(\alpha_{s}(f(t)))w_{st}(\xi)d\mu(t) \\
&=\int_{G}w_{s}\pi_2(f(t))w_{s^{-1}}w_{st}(\xi)d\mu(t)\\
&=w_{s}\circ \pi_2\rtimes w(f)(\xi).\end{aligned}$$
Hence we have \begin{align}\label{Eq:norm}
\pi_2\rtimes w(i_{G}(s)f))=w_{s}\circ \pi_2\rtimes w(f).
\end{align}
Since $w_{s}$ is an isometry, by the universal property of $F^{p}(G,A,\alpha)$ (see \cite[Theorem 3.6]{N. C. Phillips Lp}), we have
$||i_{G}(s)f||=||f||.$
Thus we can extend $i_{G}(s)$ to an isometric map from $F^{p}(G,A,\alpha)$ to itself.

For $f,g\in C_{c}(G,A,\alpha)$ and $r,t\in G$, we have $$\begin{aligned} (i_{G}(r)f)*g(t) &=
 \int_{G} i_{G}(r)f(s)\alpha_{s}(g(s^{-1}t))d\mu(s)
  \\ &=\int_{G}\alpha_{r}(f(r^{-1}s))\alpha_{s}(g(s^{-1}t))d\mu(s)
  \\ &=\int_{G} \alpha_{r}(f(s))\alpha_{rs}(g(s^{-1}r^{-1}t)) d\mu(s)
  \\ &=\alpha_{r}(\int_{G}f(s)\alpha_{s}(g(s^{-1}r^{-1}t))d\mu(s)
  \\ &=\alpha_{r}(f*g(r^{-1}t))
  \\ &=i_{G}(r)(f*g)(t).
  \end{aligned}$$
Since $i_{G}(r)$ is isometric on $F^p(G,A,\alpha)$ and $C_c(G,A,\alpha)$ is dense in $F^p(G,A,\alpha)$, it follows that $i_{G}(r)\in M_{L}(F^{p}(A,G,\alpha))$. It is easy to check $i_{G}(rs)=i_{G}(r)i_{G}(s)$ and $i_G(r)^{-1}=i_G(r^{-1})$. Then $i_G$ is an invertible isometry-valued homomorphism.

To show that $i_{G}$ is injective, let $\pi_{0}:A\rightarrow \mathcal{B}(E_{0})$ be a faithful nondegenerate  representation of $A$ on an $L^{p}$ space $E_0$, and let $(\pi_3,\lambda_{p}^{E_{0}})$ be its associated regular covariant representation on $L^{p}(G,E_{0})$. Let $\overline{\pi_3\rtimes\lambda_{p}^{E_{0}}}:M_{L}(F^{p}(G,A,\alpha))\rightarrow \mathcal{B}(L^{p}(G,E_{0}))$ be the  canonical extension of $\pi_3\rtimes\lambda_{p}^{E_{0}}$.
 By (\ref{Eq:norm}), for each $f\in C_{c}(G,A,\alpha)$,
we have
$\pi_3\rtimes \lambda_{p}^{E_{0}}(i_{G}(s)f))=\lambda_{p}^{E_{0}}(s)\circ \pi_3\rtimes \lambda_{p}^{E_{0}}(f).$
Since there is a two-sided contractive
approximate identity in $F^p(G,A,\alpha)$ (see \cite[Theorem 5.5]{Gardella and Thiel convolution}) and $(\pi_3, \lambda_{p}^{E_{0}})$ is nondegenerate,
we have $\lambda_{p}^{E_{0}}(s)=\overline{\pi_3\rtimes\lambda_{p}^{E_{0}}}\circ i_{G}(s)$.
So, if $s\neq e$, then we certainly have $\lambda_{p}^{E_{0}}(s)\neq I_{L^{p}(G,E_{0})}$, where $I_{L^{p}(G,E_{0})}$ is the identity operator on $L^{p}(G,E_{0})$. This shows that $i_{G}(s)\neq I$ if $s\neq e$, where $I$ is the identity element in $M_L(F^p(G,A,\alpha))$.

It is clear that $\overline{L}\circ i_{G}$ is a representation of $G$ and $\overline{L}\circ i_{A}$ is a representation of $A$. Let $(\pi,v)=(\overline{L}\circ i_{A},\overline{L}\circ i_{G})$.
Now it suffices to check that $(\pi,v)$ is the nondegenerate $\sigma$-finite contractive covariant representation of $(G,A,\alpha)$ and $L=\pi\rtimes  v$.

{\bf Step 3.} Verification.


For $f\in C_{c}(G,A,\alpha)$, $s,t\in G$ and $a\in A$, note that
$$\begin{aligned}  i_{G}(s)i_{A}(a)i_{G}(s^{-1})f(t) &= \alpha_{s}(i_{A}(a)i_{G}(s^{-1})f(s^{-1}t))=
\alpha_{s}(ai_{G}(s^{-1})f(s^{-1}t))\\ &=
\alpha_{s}(a) \alpha_{s}(\alpha_{s^{-1}}(f(ss^{-1}t)))=
i_{A}(\alpha_{s}(a))f(t) \end{aligned}$$
and
$$\begin{aligned} (\overline{L}\circ i_{G})(s) (\overline{L}\circ i_{A})(a) (\overline{L}\circ i_{G})(s^{-1})  =
\overline{L}\circ i_{G}(s)i_{A}(a)i_{G}(s^{-1})  =
\overline{L}\circ i_{A}(\alpha_{s}(a)).\end{aligned}$$
Since $L$ is $\sigma$-finite and contractive, it is easy to check that $\overline{L}\circ i_{A}$ is also $\sigma$-finite and contractive.
Now it remains to show that $\overline{L}\circ i_{A}$ is nondegenerate. If $\{e_{i}\}$ is a left contractive approximate identity of $A$, then $\{i_{A}(e_{i})\}$ is a left contractive approximate identity in $M_{L}(F^{p}(G,A,\alpha))$.
 For $f\in C_{c}(G,A,\alpha)$ and $\xi\in E$, we have
 $$\begin{aligned}  \lim_{i}(\overline{L}\circ i_{A}(e_{i}))L( f)\xi&=
\lim_{i}(\overline{L}(e_{i} f)\xi  =
\lim_{i} L(e_{i} f)\xi =
L(f)\xi.
\end{aligned}$$
Since $L$ is nondegenerate, it follows that $\{\overline{L}\circ i_{A}(e_{i})\}$ converges to the identity operator $I_{E}$ in $\mathcal{B}(E)$.
Hence $\pi=\overline{L}\circ i_{A}$ is nondegenerate.
Thus $(\pi,v)$ is a nondegenerate covariant representation of $(G,A,\alpha)$ on the $L^{p}$ space $E$.
\par
For $a\in A$ and $g\in C_{c}(G)$, note that
$$\begin{aligned}  (\overline{L}\circ i_{A})\rtimes (\overline{L}\circ i_{G})(i_{A}(a)i_{G}(g))&=
\int_{G}\overline{L}\circ i_{A}(i_{A}(a)g(s))\overline{L}\circ i_{G}(s)d\mu(s) \\ &=
\int_{G}\overline{L}\circ i_{A}(ag(s))\overline{L}\circ i_{G}(s)d\mu(s)\\ &=
\overline{L}\circ i_{A}(a)\int_{G}g(s)\overline{L}\circ i_{G}(s)d\mu(s) \\&=
\overline{L}\circ i_{A}(a)\overline{L}\circ i_{G}(g)=
\overline{L}(i_{A}(a)i_{G}(g))\\&=
L(i_{A}(a)i_{G}(g)).
\end{aligned}$$
By \cite[Lemma 1.87]{Williams}, the span of elements of the form $i_{A}(a)i_{G}(g)$ are dense in $F^{p}(G,A,\alpha)$ (in fact, $i_{A}(a)i_{G}(g)(s)=g(s)a$). It follows that $L=(\overline{L}\circ i_{A})\rtimes (\overline{L}\circ i_{G})=\pi\rtimes v$.
\end{proof}

\subsection{Unique $L^p$ operator matrix norms}
$\\$
$\\$
In this subsection, we shall prove two properties of $L^p$ operator algebras with unique $L^p$ operator matrix norms.

\begin{defn}[{\cite[Definition 3.3]{Lp AF}}]\label{P-matrix-norm}
Let $(X,\mathcal{B},\mu)$ be a measure space, and let $A\subseteq \mathcal{B}(L^p(X,\mu))$ be a closed subalgebra. We equip $M_n(A)$ with the matrix norms coming from the identification of $M_n(A)$ with a closed subalgebra of $\mathcal{B}(L^p(\{1,2,...,n\}\times X,\nu\times \mu))$, where $\nu$ is the counting measure on $\{1,2,...,n\}$.
\end{defn}


\begin{defn}[{\cite[Definition 4.1]{Lp AF}}]\label{matrix norm}
Let $p\in [1,\infty)$ and $A$ be a separable $L^p$ operator algebra. We say that $A$ has
{\it unique $L^p$ operator matrix norms} if whenever $(X,\mathcal{B},\mu)$ and $(Y,\mathcal{C},\nu)$ are $\sigma$-finite measure spaces such that $L^p(X,\mu)$ and $L^p(Y,\nu)$ are separable, $\pi:A\rightarrow \mathcal{B}(L^p(X,\mu))$ and $\sigma:A\rightarrow \mathcal{B}(L^p(Y,\nu))$ are isometric representations, and $\pi(A)$ and $\sigma(A)$ are given the matrix normed structure of Definition \ref{P-matrix-norm}, then $\sigma\circ\pi^{-1}:\pi(A)\rightarrow \sigma(A)$ is completely isometric.
\end{defn}

$M_n^p$ and $C(X)$ are basic examples with unique $L^p$ operator matrix norms (see \cite[Corollary 4.4 and Proposition 4.6]{Lp AF}), where $X$ is compact metrizable space.

Given a countable set $X$ endowed with the counting measure, we denote by $l^p(X)$ the corresponding $L^p$ space. When $X=\mathbb{N}$, we simply write $l^p$ for $l^p(\mathbb{N})$.

\begin{lem}\label{isometric tensor}
Let $B\subseteq \mathcal{B}(l^p)$ be a closed subalgebra, and let $A$ be a separable $L^p$ operator algebra with unique $L^p$ operator matrix norms. Then the norm on spatial $L^p$ operator tensor product $B\otimes_p A$ is independent of the choice of isometric representation of $A$.
\end{lem}

\begin{proof}
For each $n\geq 0$, we let $P_n$ be the canonical projection on $l^p$ defined by $P_n(\sum_{i=0}^{\infty}\alpha_i e_i)=\sum_{i=0}^{i=n}\alpha_i e_i$, where $\{e_i\}$ is the standard basis in $l^p$ and $\alpha_i\in \mathbb{C}$.
Then we have $\|P_nx-x\|\rightarrow 0$ for all $x\in l^p$. It is not hard to show that for every $T\in \mathcal{B}(l^p\otimes_p E)$, where $E$ is a separable $L^p$ space, we have
$$\|T\|=\sup_n\{\|(P_n\otimes I_E)T(P_n\otimes I_E)\|\},$$
where $I_E$ is the identity operator on $E$.
Since $A$ is separable, by \cite[Proposition 1.25]{N. C. Phillips Lp} and \cite[Lemma 2.7]{Lp AF}, we only need separable $L^p$ spaces and the measure is $\sigma$-finite.
Let $\pi_1:A\rightarrow \mathcal{B}(L^p(X,\mu))$ and  $\pi_2:A\rightarrow \mathcal{B}(L^p(Y,\nu))$ be two isometric representations of $A$ on two separable $L^p$ spaces and the measures $\mu,~\nu$ are $\sigma$-finite, respectively.
Thus, if $\sum b_i\otimes a_i\in B\otimes_{alg} A$ $(\otimes_{alg}$ is the algebraic tensor product), then
$$\|\sum b_i\otimes \pi_1(a_i)\|=\sup_{n}\{\|\sum(P_n b_i P_n)\otimes \pi_1(a_i)\|\}$$
and $$\|\sum b_i\otimes \pi_2(a_i)\|=\sup_{n}\{\|\sum(P_n b_i P_n)\otimes \pi_2(a_i)\|\}.$$
Since $P_n\mathcal{B}(l^p)P_n$ is isomorphic to $M_n^p$, and $A$ has unique $L^p$ operator matrix norms, we have $\|\sum(P_n b_i P_n)\otimes \pi_1(a_i)\|=\|\sum(P_n b_i P_n)\otimes \pi_2(a_i)\|$ for each $n$.
Therefore $\|\sum b_i\otimes \pi_1(a_i)\|=\|\sum b_i\otimes \pi_2(a_i)\|$.
\end{proof}

The preceding lemma shows that if $A$ is a separable $L^p$ operator algebra with unique $L^p$ operator matrix norms, and $G$ is a countable discrete group, then the norm on spatial $L^{p}$ operator tensor product $\mathcal{K}(l^{p}(G))\otimes_{p}A$ is independent of the choice of isometric representation of $A$.

%

If $A$ is a $C^*$-algebra, then it is well known that the reduced crossed product $A\rtimes_{\alpha,r}G$ does not depend on the choice of faithful representation of $A\subseteq \mathcal{B}(H)$, since there is a unique $C^*$-norm on $M_n(A)$ (see \cite[Proposition 8.4]{Brown and Ozawa}).
For $L^p$ operator algebras with unique $L^p$ operator matrix norms, we have an analogous result.

\begin{lem} \label{does not depend}
Let $(G,A,\alpha)$ be an $L^p$ operator algebra dynamical system, where $G$ is a countable discrete group and
$A$ is a separable $L^p$ operator algebra with unique $L^p$ operator matrix norms.
Let $\pi_0:A\rightarrow \mathcal{B}(E_0)$ be a nondegenerate $\sigma$-finite isometric representation of $A$ on an $L^p$ space $E_0$, and let $(\pi,\lambda_{p}^{E_0})$ be its associated regular covariant representation. Then the integrated form $\pi\rtimes\lambda_{p}^{E_0}$ is an isometric representation of $F^p_\lambda(G,A,\alpha)$.
\end{lem}


\begin{proof}
We give the proof by proving two claims.

{\bf Claim 1.} For any $f\in C_c(G,A,\alpha)$, it holds that $$\|f\|_{F^{p}_{\lambda}(G,A,\alpha)}=\sup\{\|(\pi\rtimes v(f)\|:(\pi,v)\in  \mathrm{IsoRegRep}_p(G,A,\alpha)\},$$ where $\mathrm{IsoRegRep}_p(G,A,\alpha)$ denotes the class consisting of nondegenerate $\sigma$-finite isometric regular covariant representation of $(G,A,\alpha)$.

By \cite[Lemma 3.19]{N. C. Phillips Lp}, there is a nondegenerate $\sigma$-finite isometric representation $\rho_0:A\rightarrow \mathcal{B}(L^p(X,\mu))$ such that, with $(\rho,\lambda_p^{L^p(X,\mu)})$ being its associated regular covariant representation,  the representation $\rho\rtimes \lambda_p^{L^p(X,\mu)}$ is nondegenerate and isometric on $F^p_\lambda(G,A,\alpha)$. Then one can see Claim 1.

{\bf Claim 2.} The norm of $\pi \rtimes\lambda_{p}^{E_0}(f)$ is independent of the choice of nondegenerate $\sigma$-finite isometric representation $\pi_0$.

Since $A$ is separable, by \cite[Proposition 1.25]{N. C. Phillips Lp} and \cite[Lemma 2.7]{Lp AF}, we may assume that $E_0$ is a separable $L^p$ space and the measure from $E_0$ is $\sigma$-finite.
Arbitrarily choose a nonempty finite set $F\subseteq G$. Let $P\in \mathcal{B}(l^p(G))$ be the finite rank projection onto the span of $\{\delta_{t}:t\in F\}$, where $\{\delta_{t}\}_{t\in G}$ is the canonical basis of $l^p(G)$.
 Let $I$ be the identity operator in $\mathcal{B}(E_0)$. Rather than compute the norm of $f\in C_c(G,A,\alpha)$ under the representation $\pi\rtimes \lambda_p^{E_0}:C_c(G,A,\alpha)\rightarrow \mathcal{B}(l^p(G)\otimes_p E_0)$, we will cut by the projection $P\otimes I$ and show that the norm of $(P\otimes I)\pi\rtimes \lambda_p^{E_0} (f)(P\otimes I)$ is independent of the isometric representation $\pi_0:A\rightarrow \mathcal{B}(E_0)$.
 By taking a limit over finite sets in $G$, we conclude that so is the norm of $\pi\rtimes \lambda_p^{E_0}(f)$.

Since $F$ is a finite subset of $G$, let $M_F^p$ be the collection of those $T\in \mathcal{B}(l^p(G))$ satisfying
$T(l^p(G\setminus F))=\{0\}$ and $T(l^p(F))\subseteq l^p(F)$ (see \cite[Example 1.6]{N. C. Phillips Lp}).
Let $\{e_{s,t}\}_{s,t\in F}$ be the canonical matrix units of $P\mathcal{B}(l^p(G))P\cong M_F^p$ and fix arbitrary $a\in A$.
 We claim that $$(P\otimes I)\pi(a)=(P\otimes I) \pi(a)(P\otimes I)=\sum_{t\in F}e_{t,t}\otimes \pi_0(\alpha_{t^{-1}}(a)).$$
This is clear if $\pi(a)$ is considered as a diagonal matrix in $\mathcal{B}(\bigoplus_{t\in G} E_0)$.
 Identifing $l^p(G)\otimes_p E_0$ with the $L^p$ direct sum $\bigoplus_{t\in G}E_0$, we may simply take the $L^p$ direct sum representation $$\pi(a)=\bigoplus_{t\in G}\pi_0(\alpha_{t^{-1}}(a))\in \mathcal{B}(\bigoplus_{t\in G}E_0).$$
In the spatial $L^p$ operator tensor product picture, we have $$\pi(a)=\sum_{t\in G}e_{t,t}\otimes \pi_0(\alpha_{t^{-1}}(a)),$$
where the convergence is in the strong operator topology.

Let $\lambda_{p}:G\rightarrow \mathcal{B}(l^p(G))$ be the left-regular representation. For each $s\in G$, we have $\lambda_{p}^{E_{0}}(s) =\lambda_{p}(s)\otimes I$.
Thus one can see that $$\begin{aligned}
(P\otimes I)\pi(a)(\lambda_p^{E_0}(s))(P\otimes I) &=(P\otimes I)\pi(a)(\lambda_p(s)\otimes I)(P\otimes I)
 \\&=
 (\sum_{t\in F}e_{t,t}\otimes \pi_0(\alpha_{t^{-1}}(a)))(P\lambda_p(s)P\otimes I)
  \\ &=(\sum_{t\in F}e_{t,t}\otimes \pi_0(\alpha_{t^{-1}}(a)))
   (\sum_{r\in F\cap sF}e_{r,s^{-1}r}\otimes I)
  \\ &=\sum_{r\in F\cap sF}e_{r,s^{-1}r  \otimes  \pi_0(\alpha_{r^{-1}}(a))}
  \in M_F^p\otimes \pi_0(A). \end{aligned}$$
Now if $f=\sum_{s\in G} a_s\delta_s\in C_c(G,A,\alpha)$, then we have
$$(P\otimes I)\pi\rtimes \lambda_p^{E_0}(f)(P\otimes I)=\sum_{s\in G}\sum_{r\in F\cap sF}e_{r,s^{-1}r}\otimes \pi_0(\alpha_{r^{-1}}(a_s))\in M_F^p\otimes \pi_{0}(A).$$
Since $A$ has unique $L^p$ operator norms, the norm of $(P\otimes I)\pi\rtimes\lambda_p^{E_0}(f)(P\otimes I)$ is independent of the choice of nondegenerate $\sigma$-finite isometric representation $\pi_0:A\rightarrow \mathcal{B}(E_0)$.
This proves Claim 2.

By Claims 1 and 2, $\pi\rtimes\lambda_{p}^{E_0}$ is an isometric representation of $F^p_\lambda(G,A,\alpha)$.
\end{proof}

\section{Proof of Theorem \ref{T:main}}
$\\$
In the rest of this paper, let $G$ be a countable discrete Abelian group. Thus we may assume that
the Haar measure on $G$ is the counting measure. Since the dual group $\hat{G}$ is compact, we also denote by $\mu$ the unique normalized positive Haar measure on $\hat{G}$.

%
%

\subsection{The isometry property of $\Phi_{1}$}
$\\$
$\\$
The aim of this subsection is to prove the following result.

\begin{prop}\label{change}
Let $\Phi_{1}:F^{p}(\hat{G},F^{p}(G,A,\alpha),\hat{\alpha})\rightarrow
F^{p}(G,F^{p}(\hat{G},A,\beta),\hat{\beta}\otimes\alpha)$ be the map defined in (\ref{Phi1}) under the assumption of Theorem \ref{T:main}. Then $\Phi_{1}$ is an isometric isomorphism.
\end{prop}

We first make some preparation.
For the reader's convenience, we recall the definitions of $\hat{\alpha}$ and $\hat{\beta}\otimes\alpha$. For each $\gamma\in \hat{G}$, $\hat{\alpha}_{\gamma}:C_{c}(G,A,\alpha)\rightarrow C_{c}(G,A,\alpha)$ is defined by $\hat{\alpha}_{\gamma}(f)(s):=\overline{\gamma(s)}f(s)$ for all $f\in C_{c}(G,A,\alpha)$ and $s\in G$. Then $\hat{\alpha}_\gamma$ extends to an isometry on $F^p(G,A,\alpha)$.
Let $\beta$ be the trivial action of $\hat{G}$ on $A$. For each $s\in G$, we define an isometric isomorphism $ \hat{\beta}_{s}:C(\hat{G},A,\beta)\rightarrow C(\hat{G},A,\beta)$ by $\hat{\beta}_{s}(\varphi)(\gamma):=\gamma(s)\varphi(\gamma)$ for all $\varphi\in C(\hat{G},A,\beta)$ and $\gamma\in \hat{G}$. Then we can form the tensor product action of $G$ on $F^{p}(\hat{G},A,\beta)$ by
 $(\hat{\beta}\otimes\alpha)_{s}(\varphi)(\gamma):=\gamma(s)\alpha_{s}(\varphi(\gamma)).$ Hence $(\hat{\beta}\otimes \alpha)_{s}$ extends to an element of $\mathrm{Aut}(F^{p}(\hat{G},A,\beta))$.

\begin{lem}\label{dense}
Let $(G,A,\alpha)$ be an $L^p$ operator algebra dynamical system, where $G$ is a countable discrete Abelian group, and $A$ is an $L^p$ operator algebra. Then
\begin{enumerate}
\item[(i)] $C_{c}(\hat{G}\times G,A)$ is a dense subalgebra of $F^{p}(\hat{G},F^{p}(G,A,\alpha),\hat{\alpha})$.
\item[(ii)] $C_{c}(G\times\hat{G},A)$ is a dense subalgebra of $F^{p}(G,F^{p}(\hat{G},A,\beta),\hat{\beta}\otimes\alpha)$.
\end{enumerate}
\end{lem}

\begin{proof}
(i) If $F\in C_{c}(\hat{G}\times G,A)$, then $\lambda_{F}(\gamma)(s):=F(\gamma,s)$ defines an element $\lambda_{F}\in C_{c}(\hat{G},C_{c}(G,A))$ $\subseteq C_{c}(\hat{G},F^{p}(G,A,\alpha))$ for all $\gamma\in \hat{G}$ and $s\in G$. That is, $\lambda_{F}(\gamma)\in C_{c}(G,A)$ for all $\gamma\in \hat{G}$. For $\sigma\in \hat{G}$ and $F_{1},F_{2}\in C_{c}(\hat{G}\times G,A)$, note that $$\lambda_{F_{1}}*\lambda_{F_{2}}(\sigma)=\int_{\hat{G}}\lambda_{F_{1}}(\gamma)*\hat{\alpha}_
{\gamma}(\lambda_{F_{2}}(\gamma^{-1}\sigma))d\mu(\gamma)$$ takes value in $F^{p}(G,A,\alpha)$; moreover,
by \cite[Lemma 1.108]{Williams}, the former integral takes value in $C_{c}(G,A)\subseteq F^{p}(G,A,\alpha)$.
For each $t\in G$, we have
$$\begin{aligned}
\lambda_{F_{1}}*\lambda_{F_{2}}(\sigma)(t)
 &=\int_{\hat{G}}\sum_{s\in G}\lambda_{F_{1}}(\gamma)(s)\alpha_{s}(\hat{\alpha}_{\gamma}
 (\lambda_{F_{2}}(\gamma^{-1}\sigma))(s^{-1}t))d\mu(\gamma) \\
 &=\int_{\hat{G}}\sum_{s\in G}\lambda_{F_{1}}(\gamma)(s) \overline{\gamma(s^{-1}t)}\alpha_{s}(\lambda_{F_{2}}(\gamma^{-1}\sigma)(s^{-1}t))d\mu(\gamma).
\end{aligned}$$
For convenience, we can rewrite this equality as
$$F_{1}*F_{2}(\sigma,t)=\int_{\hat{G}}\sum_{s\in G} F_{1}(\gamma,s)\overline{\gamma(s^{-1}t)}\alpha_{s}(F_{2}(\gamma^{-1}\sigma,s^{-1}t))d\mu(\gamma).$$
Therefore we can view $C_{c}(\hat{G}\times G,A)$ as a subalgebra of $F^{p}(\hat{G},F^{p}(G,A,\alpha),\hat{\alpha})$.

Since $C_c(G,A,\alpha)$ is dense in $F^p(G,A,\alpha)$,
by \cite[Lemma 1.87]{Williams}, $C(\hat{G},C_c(G,A,\alpha),\hat{\alpha})$ is dense in $C(\hat{G},F^p(G,A,\alpha),\hat{\alpha})$.
Thus $C_{c}(\hat{G}\times G,A)$ is dense in $F^{p}(\hat{G},F^{p}(G,A,\alpha),\hat{\alpha})$.

Using a similar argument (i) as in the proof of Lemma \ref{dense}, one can prove (ii) by exchanging $G$ with $\hat{G}$.
\end{proof}

%
%
%
%

%

\begin{proof}[Proof of Proposition \ref{change}]
The proof consists of four claims.

{\bf Claim 1.} $\Phi_{1}$ is a homomorphism from  $C_{c}(\hat{G}\times G,A)$ onto $C_{c}(G\times \hat{G},A)$.

Obviously, $\Phi_1$ is a surjective map. Next, we will show that $\Phi_1$ is a homomorphism.
For each $F_{1},F_{2}\in C_{c}(\hat{G}\times G,A)$, we have
$$\lambda_{\Phi_{1}(F_{1})}*\lambda_{\Phi_{1}(F_{2})}(t)=
\sum_{s\in G}\lambda_{\Phi_{1}(F_{1})}(s)*(\hat{\beta}\otimes\alpha)_{s}(\lambda_{\Phi_{1}(F_{2})}(s^{-1}t)).$$
Then we have

$\begin{aligned}
\Phi_{1}(F_{1})*\Phi_{1}(F_{2})(t,\sigma)
&=\lambda_{\Phi_{1}(F_{1})}*\lambda_{\Phi_{1}(F_{2})}(t)(\sigma)\\
 &=\sum_{s\in G}\int_{\hat{G}}\lambda_{\Phi_{1}(F_{1})}(s)(\gamma)\beta_{\gamma}( (\hat{\beta}\otimes\alpha)_{s}(\lambda_{\Phi_{1}(F_{2})}(s^{-1}t))(\gamma^{-1}\sigma))d\mu(\gamma)\\
 &=\sum_{s\in G}\int_{\hat{G}}\lambda_{\Phi_{1}(F_{1})}(s)(\gamma) (\gamma^{-1}\sigma)(s)
 \alpha_{s}(\lambda_{\Phi_{1}(F_{2})}(s^{-1}t)(\gamma^{-1}\sigma))d\mu(\gamma)\\
 &=\sum_{s\in G}\int_{\hat{G}} \gamma(s)F_{1} (\gamma,s) (\gamma^{-1}\sigma)(s) (\gamma^{-1}\sigma)(s^{-1}t)\alpha_{s}(F_{2} (\gamma^{-1}\sigma,s^{-1}t))d\mu(\gamma)\\
&=\sum_{s\in G}\int_{\hat{G}} F_{1}(\gamma,s) \overline{\gamma(s^{-1}t)}\sigma(t)\alpha_{s}(F_{2}(\gamma^{-1}\sigma,s^{-1}t))d\mu(\gamma)\\
&=\int_{\hat{G}}\sum_{s\in G}F_{1}(\gamma,s) \overline{\gamma(s^{-1}t)}\sigma(t)\alpha_{s}(F_{2}(\gamma^{-1}\sigma,s^{-1}t))d\mu(\gamma)\\
&=\sigma(t) \lambda_{F_{1}}*\lambda_{F_{2}}(\sigma)(t)=\Phi_{1}(F_{1}*F_{2})(t,\sigma).\end{aligned}$\\
This shows that $\Phi_{1}(F_{1}*F_{2})=\Phi_{1}(F_{1})*\Phi_{1}(F_{2}).$
Hence $\Phi_{1}$ is a homomorphism from $C_{c}(\hat{G}\times G,A)$ onto $C_{c}(G\times\hat{G},A)$.

Since $\Phi_{1}$ maps a dense subalgebra onto a dense subalgebra, it suffices to show that $\Phi_{1}$ is isometric for the universal norms.

Let $L$ be a nondegenerate $\sigma$-finite contractive representation of $F^{p}(G,F^{p}(\hat{G},A,\beta),\hat{\beta}\otimes \alpha)$. By Theorem~\ref{correspondence},
there exists a nondegenerate $\sigma$-finite contractive covariant representation $(R,U)$ of $(G,F^{p}(\hat{G},A,\beta),\hat{\beta}\otimes\alpha)$ such that $L=R\rtimes U$, and there exists a nondegenerate $\sigma$-finite contractive covariant representation $(\pi,V)$ of $(\hat{G},A,\beta)$ such that $R=\pi\rtimes V$.

{\bf Claim 2.}  $U_{s}V_{\gamma}=\gamma(s)V_{\gamma}U_{s}$ for each $\gamma\in \hat{G}$ and each $s\in G$.

Let $i_{\hat{G}}$ be the canonical map of $\hat{G}$ into
$M_{L}(F^{p}(\hat{G}))$. Then
$\begin{aligned} i_{\hat{G}}(\gamma)\varphi(\sigma) &=\beta_{\gamma}(\varphi(\gamma^{-1}\sigma)) =\varphi(\gamma^{-1}\sigma)\end{aligned}$
for all $\gamma,\sigma\in \hat{G}$ and $\varphi\in C(\hat{G})$. For each $s\in G$,
 $$\begin{aligned} \hat{\beta}_{s}\circ i_{\hat{G}}(\gamma)\varphi(\sigma)  &=
\sigma(s)i_{\hat{G}}(\gamma)\varphi(\sigma)  =\sigma(s)\varphi(\gamma^{-1}\sigma) =\gamma(s)(\gamma^{-1}\sigma)(s)\varphi(\gamma^{-1}\sigma)\\
&=\gamma(s)\hat{\beta}_{s}\varphi(\gamma^{-1}\sigma) =\gamma(s)i_{\hat{G}}(\gamma)\circ \hat{\beta}_{s}\varphi(\sigma),\end{aligned}$$
it follows that
$\hat{\beta}_{s}\circ i_{\hat{G}}(\gamma)=\gamma(s)i_{\hat{G}}(\gamma)\circ \hat{\beta}_{s}.$

Now if $a\in A$, $\varphi\in C(\hat{G})$ and $f\in C_{c}(G)$, then the functions which are elementary tensors of the form $\varphi\otimes f\otimes a$ span a dense subset of $C_{c}(\hat{G}\times G,A)$.
Since
$$\begin{aligned} U_{s}V_{\gamma}L(\varphi\otimes f\otimes a) &= U_{s}V_{\gamma} R\rtimes U(\varphi\otimes f\otimes a) =U_{s}V_{\gamma}\pi(a)V(\varphi)U(f)\\
& =U_{s}\pi(a)V_{\gamma}V(\varphi)U(f)
=U_{s}\pi(a)V(i_{\hat{G}}(\gamma)\varphi)U(f) \\
 &=U_{s}\pi\rtimes V(i_{\hat{G}}(\gamma)(\varphi)\otimes a)U(f)
 =\pi\rtimes V((\hat{\beta}\otimes \alpha)_{s}(i_{\hat{G}}(\gamma)(\varphi)\otimes a))U_{s}U(f)\\
 &=\pi(\alpha_{s}(a))V(\hat{\beta}_{s}\circ i_{\hat{G}}(\gamma)\varphi)U_{s}U(f)
 =\pi(\alpha_{s}(a))V(\gamma(s)i_{\hat{G}}(\gamma)\circ \hat{\beta}_{s}(\varphi))U_{s}U(f)\\
 &=\gamma(s)\pi(\alpha_{s}(a))V_{\gamma}V(\hat{\beta}_{s}(\varphi))U_{s}U(f)
 =\gamma(s)V_{\gamma}\pi(\alpha_{s}(a))V(\hat{\beta}_{s}(\varphi))U_{s}U(f)\\
 &=\gamma(s)V_{\gamma}U_{s}\pi\rtimes V((\hat{\beta}\otimes\alpha)_{s^{-1}}(\hat{\beta}_{s}(\varphi)\otimes \alpha_{s}(a))U(f)\\
&=\gamma(s)V_{\gamma}U_{s}\pi(a)V(\varphi)U(f) =\gamma(s)V_{\gamma}U_{s}L(\varphi\otimes f\otimes a), \end{aligned}$$
 and $L$ is nondegenerate, it follows that \begin{align}\label{eq:exchange}U_{s}V_{\gamma}=\gamma(s)V_{\gamma}U_{s}.\end{align}

{\bf Claim 3.} $(\pi,U)$ is a $\sigma$-finite covariant representation of $(G,A,\alpha)$ and $(\pi\rtimes U,V)$ is a $\sigma$-finite covariant representation of $(\hat{G},F^{p}(G,A,\alpha),\hat{\alpha})$.

For each $b\in A$,
a similar computation as in the proof of (\ref{eq:exchange}) shows that
$$\begin{aligned} U_{s}\pi(b)\pi(a)V(\varphi)U(f) &= U_{s}\pi(ba)V(\varphi)U(f) \\
 &=U_{s}\pi\rtimes V(\varphi\otimes ba)U(f)\\
 &=\pi\rtimes V((\hat{\beta}\otimes\alpha)_{s}(\varphi\otimes ba))U_{s}U(f)\\
 &=\pi(\alpha_{s}(b))\pi\rtimes V(\hat{\beta}_{s}(\varphi)\otimes \alpha_{s}(a))U_{s}U(f)\\
 &=\pi(\alpha_{s}(b))U_{s}\pi\rtimes V(\varphi\otimes a)U(f)\\
 &=\pi(\alpha_{s}(b))U_{s}\pi(a)V(\varphi)U(f). \end{aligned}$$
Thus $(\pi,U)$ is a covariant representation of $(G,A,\alpha)$. Obviously, the representation $\pi$ is $\sigma$-finite.

Since the action of $\hat{G}$ on $A$ is trivial, by (\ref{eq:exchange}),  we have
$$\begin{aligned} V_{\gamma}\pi\rtimes U(a\otimes f)  =\pi(a) V_{\gamma}\sum_{s\in G} f(s)U_{s} =\pi(a)\sum_{s\in G} f(s)\overline{\gamma(s)}U_{s}V_{\gamma} =\pi\rtimes U(\hat{\alpha}_{\gamma}(a\otimes f))V_{\gamma}. \end{aligned}$$
Hence $(\pi\rtimes U,V)$ is a covariant representation of $(\hat{G},F^{p}(G,A,\alpha),\hat{\alpha})$.
Since $U$ is $\sigma$-finite, it follows that $\pi\rtimes U$ is $\sigma$-finite.

 {\bf Claim 4.} $||\Phi_{1}(F)||\leq ||F||.$

 From Claim 3, $L'=(\pi\rtimes U)\rtimes V$ is a contractive representation of $F^{p}(\hat{G},F^{p}(G,A,\alpha),\hat{\alpha})$.
 By the Fubini theorem and (\ref{eq:exchange}), we have
 $$\begin{aligned} L(\Phi_{1}(F)) &= \sum_{s\in G} \pi\rtimes V(\lambda_{\Phi_{1}(F)}(s))U_{s} =\sum_{s\in G}\int_{\hat{G}}\pi(\Phi_{1}(F)(s,\gamma))V_{\gamma}U_{s}d\mu(\gamma)\\
 &=\sum_{s\in G}\int_{\hat{G}}\pi(F(\gamma,s))\gamma(s)V_{\gamma}U_{s}d\mu(\gamma)=\int_{\hat{G}}\sum_{s\in G}\pi(F(\gamma,s))U_{s}V_{\gamma}d\mu(\gamma)
 =L'(F). \end{aligned}$$
 It follows that $||\Phi_{1}(F)||\leq ||F||.$ Reversing the above argument gives $||F||\leq ||\Phi_{1}(F)||$. This completes the proof.
\end{proof}

\subsection{A characterization of $\Phi_{2}$ being an isomorphism}
$\\$
$\\$
The aim of this subsection is to prove the following.

\begin{prop}\label{injective}
Let $\Phi_2:F^{p}(G,
F^{p}(\hat{G},A,\beta),\hat{\beta}\otimes\alpha)\rightarrow F^{p}(G,C_{0}(G,A),\mathrm{lt}\otimes\alpha)$ be defined in (\ref{Phi2}) under the assumption of Theorem \ref{T:main}. Then
$\Phi_{2}$ is an isomorphism if and only if either $G$ is finite or $p=2$; in particular, $\Phi_{2}$ is an isometric isomorphism for $p=2$.
\end{prop}

Clearly, $\Phi_{2}$ is defined in terms of the Gelfand transformation of reduced group $L^{p}$ operator algebra.
The following theorem concerning the Gelfand transformation of $F^p_\lambda(G)$ will be useful later.

\begin{thm}[{\cite[Theorem 2.7]{Wang}} or {\cite[Proposition
3.22 and Corollary 3.20]{Gardella and Thiel}}]\label{Gelfand}
Let $G$ be a locally compact Abelian group and $p\in[1,\infty)\setminus \{2\}$. Then
\begin{enumerate}
 \item[(i)] the Gelfand transformation $\Gamma_{p}:F_{\lambda}^{p}(G)\rightarrow C_{0}(\hat{G})$ is an injective contractive map with dense range;
 \item[(ii)] the Gelfand transformation $\Gamma_{p}$ is surjective if and only if $G$ is finite.
\end{enumerate}
\end{thm}


\begin{lem}\label{extention of Gelfand transformation}
Let $A$ be a separable $L^p$ operator algebra with unique $L^p$ operator matrix norms, $G$ be a countable discrete Abelian group, and  $\varphi_2:F^{p}(\hat{G},A,\beta)\rightarrow C_{0}(G,A) $ be defined in (\ref{Phi2}).
\begin{enumerate}
\item[(i)] If $p\in[1,\infty)\setminus \{2\}$, then $\varphi_2$ is an isomorphism if and only if $G$ is finite.
\item[(ii)] If $p=2$, then $\varphi_2$ is an isometric isomorphism.
 \end{enumerate}
\end{lem}

\begin{proof}
We first prove two key claims.

{\bf Claim 1.} There exists a nondegenerate $\sigma$-finite isometric representation $\pi_0$ of $A$ on some $L^p$ space such that $F^{p}(\hat{G},A,\beta)$ is isometrically isomorphic to $F^{p}_{\lambda}(\hat{G})\otimes_{p} \pi_0(A)$.

Since $\hat{G}$ is amenable, it follows that $F^{p}(\hat{G},A,\beta)$ is isometrically isomorphic to $F^{p}_{\lambda}(\hat{G},A,\beta)$ (see \cite[Theorem 7.1]{Phillips look like}). By \cite[Lemma 3.19]{N. C. Phillips Lp}, there exists a a nondegenerate $\sigma$-finite isometric representation $\pi_0:A\rightarrow  \mathcal{B}(L^p(X,\nu))$ such that, with $(\pi,\lambda_{p}^{L^p(X,\nu)})$ being its associated regular covariant representation, the representation $\pi\rtimes \lambda_{p}^{L^p(X,\nu)}$ is nondegenerate and isometric on $F^p_\lambda(\hat{G},A,\beta)$. Since $\beta$ is the trivial action of $\hat{G}$ on $A$, it follows that $F^p_\lambda(\hat{G},A,\beta)$ is isometrically isomorphic to $F^p_\lambda(\hat{G})\otimes_p \pi_0(A)$. We denote by $\rho$ the isometric isomorphism from $F^p(\hat{G},A,\beta)$ onto $F^p_\lambda(\hat{G})\otimes_p \pi_0(A)$.


{\bf Claim 2.} $\varphi_2: F^{p}(\hat{G},A,\beta)\rightarrow C_{0}(G,A)$ is an injective homomorphism.

Note that the norm on $C_0(G)\otimes_p A\cong C_0(G,A)$ is independent of the choice of isometric representation of $A$, then we let $\mathrm{Id}_{C_0(G)}\otimes \pi_0$ be the isometric isomorphism from  $C_0(G)\otimes_p A$ onto $C_0(G)\otimes_p \pi_0(A)$ by sending $f\otimes a$ to $f\otimes \pi_0(a)$ for all $f\in C_0(G)$ and $a\in A$.
Hence, the homomorphism $(\mathrm{Id}_{C_0(G)}\otimes \pi_0)\circ\varphi_2\circ\rho^{-1}:F^{p}_{\lambda}(\hat{G})\otimes_{p} \pi_0(A)\rightarrow C_{0}(G)\otimes_{p}\pi_0(A)$ sends $\varphi\otimes \pi_0(a)$ to $\Gamma_p(\varphi)\otimes \pi_0(a)$ for all $\varphi\in F^{p}_{\lambda}(\hat{G})$.
Hence $(\mathrm{Id}_{C_0(G)}\otimes \pi_0)\circ\varphi_2\circ\rho^{-1}=\Gamma_p\otimes\mathrm{id}_{\pi_0(A)}$. To show $\varphi_2$ is an injective homomorphism, it suffices to show that $\Gamma_p\otimes\mathrm{id}_{\pi_0(A)}$ is an injective homomorphism.

If $\xi\in L^{p}(\hat{G}\times X,\mu\times \nu)$, then $\xi(x)(\gamma):=\xi(\gamma,x)$ defines an element $\xi(x)\in L^{p}(\hat{G},\mu)$. Since $\hat{G}$ is a compact Abelian group and $\mu(\hat{G})=1$, by \cite[Exercise 5 (a)]{Rudin}, it follows that $\xi(x)\in L^1(G,\mu)\subseteq F^{p}_\lambda(\hat{G})$.
Let $\widehat{\xi}(t,x):=\Gamma_{p}(\xi(x))(t)$. If $F=\sum_{i=1}^{m}f_{i}\otimes \pi_0(a_{i})$ and $\xi=\sum_{j=1}^{n}g_{j}\otimes h_{j}$, where $f_{i}\in C(\hat{G})$, $a_{i}\in A$, $g_{j}\in L^p(\hat{G},\mu)$ and $h_{j}\in L^p(X,\nu)$,
then a direct computation shows that
$$\begin{aligned} \Gamma_{p}\otimes\mathrm{id}_{\pi_0(A)}(F)(\widehat{\xi}) &=
(\sum_{i=1}^{m}\Gamma_{p}(f_{i})\otimes \pi_0(a_{i}))(\sum_{j=1}^{n}\Gamma_{p}(g_{j})\otimes h_{j})\\
 &=\sum_{i,j}\Gamma_{p}(f_{i}*g_{j})\otimes \pi_0(a_{i})h_{j}
 =\widehat{F(\xi)}. \end{aligned}$$
Note that the elements of the form $\sum_{i=1}^{m}f_{i}\otimes \pi_0(a_{i})$ are dense in $F^p_\lambda(\hat{G})\otimes_p \pi_0(A)$ and the elements of form $\sum_{j=1}^{n}g_{j}\otimes h_{j}$ are dense in $L^p(\hat{G}\times X,\mu\times\nu)$. Since $\Gamma_{p}\otimes\mathrm{id}_{\pi_0(A)}$ is contractive, it follows that
\begin{align}\label{Eq:tenseor}
   (\Gamma_{p}\otimes\mathrm{id}_{A})(F)(\widehat{\xi})=\widehat{F(\xi)}  \ \ \ \ \ \textup{for~all}~\ \ F\in F^{p}_\lambda(\hat{G})\otimes_{p}A, \ ~ \xi\in L^p(\hat{G}\times X,\mu\times\nu).
 \end{align}

For a proof of Claim 2 by contradiction, we assume that $\Gamma_{p}\otimes\mathrm{id}_{\pi_0(A)}(F)=0$ for some nonzero $F\in F^{p}_\lambda(\hat{G})\otimes_{p}\pi_0(A)$. Then there exists $\xi\in L^p(\hat{G}\times X)$ such that $F(\xi)\neq 0$.
By (\ref{Eq:tenseor}), we have $\Gamma_{p}\otimes \mathrm{id}_{\pi_0(A)} (F)(\widehat{\xi})=0$. This implies $\widehat{F(\xi)}=0$.
So, for $x\in X$ and $t\in G$, we have $\widehat{F(\xi)}(t,x)=\Gamma_p(F(\xi)(x))(t)=0$. Then $\Gamma_p(F(\xi)(x))=0$. Since $\Gamma_p$ is injective, we have $F(\xi)=0$, which is a contradiction. This proves Claim 2.

If $G$ is a locally compact Abelian group, then $G$ is finite if and only if its dual group $\hat{G}$ is finite (see \cite[Theorem 9.15]{Conway}). By Theorem \ref{Gelfand}, the Gelfand transformation $\Gamma_p:F^p_\lambda(\hat{G})\rightarrow C_0(G)$ is an isomorphism if and only if either $p=2$ or $G$ is finite.
If $p\in[1,\infty)\setminus\{2\}$, then $\varphi_2$ is an isomorphism if and only if $G$ is finite.

Now assume that $p=2$. Since $\Gamma_{2}:F^{2}_\lambda(\hat{G})\rightarrow C_{0}(G)$ is an isometric isomorphism, by \cite[Corollary 1.12]{Piser}, $\Gamma_{2}\otimes\mathrm{id}_{\pi_0(A)}:F^2_\lambda(\hat{G})\otimes_2 \pi_0(A)\rightarrow C_0(G)\otimes_2 \pi_0(A)$ is an isometric isomorphism.
So $\varphi_2$ is an isometric isomorphism.
\end{proof}



The following proposition will be useful in the proof for the injectivity of $\Phi_2$.

\begin{prop}[{\cite[Propositions 4.8 and 4.9]{N. C. Phillips Lp}}]\label{expection}
Let $p\in [1,\infty)$ and $G$ be a countable discrete group. Then there exist contractive linear maps $\{E_{s}:  s\in G\}$ from $F^{p}_{\lambda}(G,A,\alpha)$ to $A$ satisfying
\begin{enumerate}
\item[(i)]  for any $a\in A, s,t\in G$, $$E_s(a\delta_t)=\begin{cases}
a,& s=t,\\
0,& s\ne t,
\end{cases}$$ where $\delta_{s}$ is the function in $C_c(G)$ which is $1$ at $s$ and zero elsewhere;
\item[(ii)] for any $f\in F^{p}_{\lambda}(G,A,\alpha)$, $f=0$ if and only if $E_{s}(f)=0$ for all $s\in G$.
\end{enumerate}
\end{prop}
%
%

\begin{proof}[Proof of Proposition \ref{injective}]

We first prove two key claims.

{\bf Claim 1.} $\Phi_{2}=\varphi_2\rtimes \mathrm{id}$.

 The Gelfand transformation $\Gamma_p:F^p_\lambda(\hat{G})\rightarrow C_0(G)$ sends $\varphi\in L^1(\hat{G})$ to its Fourier transformation  $\hat{\varphi}\in C_0(G)$ which is given by $\hat{\varphi}(t)=\int_{\hat{G}} \varphi(\gamma)\overline{\gamma(t)}d\mu(\gamma)$.
Hence the homomorphism $\varphi_2:F^{p}(\hat{G},A,\beta)\rightarrow C_{0}(G,A)$ satisfies
$$\varphi_2(\psi)(t)=\int_{\hat{G}}\psi(\gamma)\overline{\gamma(t)}d\mu(\gamma)$$ for all $\psi\in C(\hat{G},A)$.
Since
$$\begin{aligned} \varphi_2((\hat{\beta}\otimes\alpha)_{s}(\psi))(t) &= \int_{\hat{G}}(\hat{\beta}\otimes\alpha)_{s}(\psi)(\gamma)\overline{\gamma(t)}d\mu(\gamma) \\
 &=\int_{\hat{G}}\gamma(s)\alpha_{s}(\psi(\gamma))\overline{\gamma(t)}d\mu(\gamma)\\
 &=\alpha_{s}(\int_{\hat{G}}\psi(\gamma)\overline{\gamma(s^{-1}t)}d\mu(\gamma))\\
 &=(\mathrm{lt}\otimes\alpha)_{s}(\varphi_2(\psi))(t), \end{aligned}$$
it follows that $\varphi_2$ is equivariant.


Let $L$ be a nondegenerate $\sigma$-finite contractive representation of $F^{p}(G,C_{0}(G,A),\rm{lt}\otimes\alpha)$. Then, by Theorem \ref{correspondence}, there exists a nondegenerate $\sigma$-finite contractive covariant representation $(\pi,v)$ of $(G,C_{0}(G,A),\rm{lt}\otimes\alpha)$ such that $L=\pi\rtimes v$.
Since $\varphi_2 (F^{p}(\hat{G},A,\beta))$ is dense in $C_{0}(G,A)$, it follows that $(\pi\circ\varphi_2,v)$ is a nondegenerate $\sigma$-finite contractive covariant representation of $(G,F^{p}(\hat{G},A,\beta),\hat{\beta}\otimes\alpha)$.
By the universal property of $F^{p}(G,F^{p}(\hat{G},A,\beta),\hat{\beta}\otimes\alpha)$ (see \cite[Theorem 3.6]{N. C. Phillips Lp}), there exists a contractive homomorphism $\varphi_2\rtimes \mathrm{id}: F^{p}(G,F^{p}(\hat{G},A,\beta),\hat{\beta}\otimes\alpha)\rightarrow
F^{p}(G,C_{0}(G,A),\rm{lt}\otimes\alpha).$
For each $F\in C_{c}(G\times \hat{G},A)$, we have $\lambda_{F}(s)\in C(\hat{G},A,\beta)\subseteq F^{p}(\hat{G},A,\beta)$ and $(\varphi_2\rtimes \mathrm{id})(F)(s)=\varphi_2(\lambda_F(s))$.
Hence
$$\begin{aligned} (\varphi_2\rtimes \mathrm{id})(F)(s,t) =
\int_{\hat{G}}\lambda_{F}(s)(\gamma)\overline{\gamma(t)}d\mu(\gamma)
 =\int_{\hat{G}}F(s,\gamma)\overline{\gamma(t)}d\mu(\gamma). \end{aligned}$$
This implies that $\Phi_{2}=\varphi_2\rtimes \mathrm{id}$.

{\bf Claim 2.} $\Phi_{2}$ is an injective homomorphism.

Since $G$ is amenable,
by \cite[Theorem 7.1]{Phillips look like},
 there are two isometric isomorphisms $\pi_1: F^{p}(G,F^{p}(\hat{G},A,\beta),\hat{\beta}\otimes\alpha)\rightarrow F^{p}_\lambda(G,F^{p}(\hat{G},A,\beta),\hat{\beta}\otimes\alpha)$ and $\pi_2: F^{p}(G,C_{0}(G,A),\mathrm{lt}\otimes\alpha)\rightarrow F^{p}_\lambda(G,C_{0}(G,A),\mathrm{lt}\otimes\alpha)$.

For each $s\in G$, let $E_s$ be the linear map as in Proposition \ref{expection}. Then we have the following commutative diagram
\[
\xymatrix{
 & F^{p}(G,F^{p}(\hat{G},A,\beta),\hat{\beta}\otimes\alpha) \ar[r]^{\Phi_2}  \ar[d]^{\pi_1}   &F^{p}(G,C_{0}(G,A),\mathrm{lt}\otimes\alpha) \ar[d]^{\pi_2}\\
  &  F^{p}_{\lambda}(G,F^{p}(\hat{G},A,\beta),\hat{\beta}\otimes\alpha) \ar[d]^{E_s}   &F^{p}_\lambda(G,C_{0}(G,A),\mathrm{lt}\otimes\alpha) \ar[d]^{E_s}  \\
&  F^{p}(\hat{G},A,\beta) \ar[r]^{\varphi_2}   & C_{0}(G,A).
}
\]
Choose $F\in F^{p}(G,F^{p}(\hat{G},A,\beta),\hat{\beta}\otimes\alpha)$ such that $\Phi_2(F)=0$.
Then we have $\varphi_2\circ E_s(\pi_1(F))=0$. Since $\varphi_2$ is injective, we get that $E_s(\pi_1(F))=0$ for all $s\in G$. By Proposition \ref{expection},
we have $\pi_1(F)=0$. This will imply that $F=0$. This proves Claim 2.

If $p\in [1,\infty)\setminus\{2\}$, then, by Lemma \ref{extention of Gelfand transformation}, $\varphi_2$ is an isomorphism if and only if $G$ is finite. Hence $\Phi_{2}$ is an isomorphism if and only if $G$ is finite.

Now assume that $p=2$. We have proved that $\Phi_2$ is a contractive homomorphism. Since $\varphi_2$ is an equivariant  isometric isomorphism, it follows that $\varphi_2^{-1}$ is also an equivariant isometric isomorphism. By the universal property of $F^{2}(G,C_{0}(G,A),\rm{lt}\otimes\alpha)$ (see \cite[Theorem 3.6]{N. C. Phillips Lp}), there exists a contractive homomorphism  $\varphi_2^{-1}\rtimes \mathrm{id}: F^{2}(G,C_{0}(G,A),\mathrm{lt}\otimes\alpha)\rightarrow F^{2}(G,
F^{2}(\hat{G},A,\beta),\hat{\beta}\otimes\alpha)$.
Then, by the Fourier inversion formula and a direct calculation similar to that in the proof of Claim 1, one can check that
$\Phi_2$ is the inverse of $\varphi_2^{-1}\rtimes \mathrm{id}$. Hence $\Phi_{2}$ is an isometric isomorphism.
\end{proof}

\subsection{The isometry property of $\Phi_3$}
$\\$
$\\$
The aim of this subsection is to prove $\Phi_3$ is an isometric isomorphism and $\Phi_3\circ\Phi_2\circ\Phi_1$ is equivariant for the
double dual action $\hat{\hat{\alpha}}$ of $G$ on $F^{p}(\hat{G},F^{p}(G,A,\alpha),\hat{\alpha})$ and
the action $(\mathrm{rt}\otimes \alpha)\otimes \mathrm{id}$ of $G$ on $F^{p}(G,C_{0}(G,A),\mathrm{lt}\otimes \mathrm{id})$.

\begin{lem}\label{change for trival action}
Let $\Phi_3:F^{p}(G,C_{0}(G,A),\mathrm{lt}\otimes\alpha)\rightarrow F^{p}(G,C_{0}(G,A),\mathrm{lt}\otimes \mathrm{id})$ be defined in (\ref{Phi3}) under the assumption of Theorem \ref{T:main}. Then $\Phi_3$ is an isometric isomorphism.
\end{lem}

\begin{proof}
We can define an isometric isomorphism $\varphi_{3}:C_{0}(G,A)\rightarrow C_{0}(G,A)$ by $\varphi_{3}(f)(t):=\alpha^{-1}_{t}(f(t)).$
Then
$$\begin{aligned} \varphi_{3}((\mathrm{lt}\otimes \alpha)_{s}(f))(t))&=
\alpha^{-1}_{t}(\alpha_{s}(f(s^{-1}t)))\\
&=\alpha^{-1}_{s^{-1}t}(f(s^{-1}t))\\
&=(\mathrm{lt}\otimes \mathrm{id})_{s}(\varphi_{3}(f))(t). \end{aligned}$$
So $\varphi_{3}$ is equivariant. By the universal property of $F^{p}(G,C_{0}(G,A),\mathrm{lt}\otimes\alpha)$ and $F^{p}(G,C_{0}(G,A),\mathrm{lt}\otimes \mathrm{id})$ (see \cite[Theorem 3.6]{N. C. Phillips Lp}), using a similar argument in Lemma \ref{injective}, one can show that $\Phi_{3}=\varphi_{3}\rtimes \mathrm{id}$ is an isometric isomorphism.
\end{proof}

For the reader's convenience, we recall the double dual action $\hat{\hat{\alpha}}$ of $G$ on $F^{p}(\hat{G},F^{p}(G,A,\alpha),\hat{\alpha})$ and the action $(\mathrm{rt}\otimes \alpha)\otimes \mathrm{id}$ of $G$ on $F^{p}(G,C_{0}(G,A),\mathrm{lt}\otimes \mathrm{id})$. The double dual action $\hat{\hat{\alpha}}$ of $G$ on $C_{c}(\hat{G}\times G,A)\subseteq F^{p}(\hat{G},F^{p}(G,A,\alpha),\hat{\alpha})$ is given by $\hat{\hat{\alpha}}_{t}(F)(\gamma, s):=\overline{\gamma(t)} F(\gamma, s)$ for all $t\in G$. Let $\mathrm{rt}$ denote the right-translation of $G$ on $C_{0}(G)$, that is,
$(\mathrm{rt})_{t}f(s):=f(st)$ for all $f\in C_{0}(G)$. Then we get an $L^{p}$ operator algebra dynamical system
$\mathrm{rt}\otimes \alpha:G\rightarrow \mathrm{Aut}(C_{0}(G,A))$, where $(\mathrm{rt}\otimes \alpha)_{t}f(s):=\alpha_{t}(f(st))$ for all $f\in C_{0}(G,A)$ and $t\in G$.
Thus we get an  $L^{p}$ operator algebra dynamical system
$(\mathrm{rt}\otimes \alpha)\otimes \mathrm{id}:G\rightarrow \mathrm{Aut}(F^{p}(G,C_{0}(G,A),\mathrm{lt}\otimes \mathrm{id})$, where $((\mathrm{rt}\otimes \alpha)\otimes \mathrm{id})_{r}F(s,t):=\alpha_{r}(F(s,tr))$ for all $F\in C_{c}(G\times G,A)\subseteq F^{p}(G,C_{0}(G,A),\mathrm{lt}\otimes \mathrm{id})$ and $r\in G$.

\begin{lem}\label{1-3}
The homomorphism
$$\Phi_{3}\circ \Phi_{2}\circ\Phi_{1}: F^{p}(\hat{G},F^{p}(G,A,\alpha),\hat{\alpha})\rightarrow  F^{p}(G,C_{0}(G,A),\mathrm{lt}\otimes \mathrm{id})$$ is equivariant for the
double dual action $\hat{\hat{\alpha}}$ of $G$ on $F^{p}(\hat{G},F^{p}(G,A,\alpha),\hat{\alpha})$ and
the action $(\mathrm{rt}\otimes \alpha)\otimes \mathrm{id}$ of $G$ on $F^{p}(G,C_{0}(G,A),\mathrm{lt}\otimes \mathrm{id})$.
\end{lem}

\begin{proof}
For $F\in C_{c}(\hat{G}\times G,A)$ and $s,t\in G$, we have
$$\begin{aligned} \Phi_{3}\circ \Phi_{2}\circ \Phi_{1}(F)(s,t)&= \alpha_{t}^{-1}(\Phi_{2}\circ \Phi_{1}(F)(s,t))\\
&=\alpha_{t}^{-1}(\int_{\hat{G}}\Phi_{1}(F)(s,\gamma)\overline{\gamma(t)}d\mu(\gamma))\\
&=\int_{\hat{G}}\alpha_{t}^{-1}(F(\gamma,s))\overline{\gamma(s^{-1}t)}d\mu(\gamma). \end{aligned}$$
Then, for each $r\in G$, we have
$$\begin{aligned} \Phi_{3}\circ \Phi_{2}\circ \Phi_{1}(\hat{\hat{\alpha}}_{r}(F))(s,t)&= \int_{\hat{G}}\alpha_{t}^{-1}(\hat{\hat{\alpha}}_{r}(F)(\gamma,s))\overline{\gamma(s^{-1}t)}d\mu(\gamma)\\
&=\int_{\hat{G}}\alpha_{t}^{-1}(F(\gamma,s))\overline{\gamma(s^{-1}tr)}d\mu(\gamma)\\
&=\alpha_{r}(\int_{\hat{G}} \alpha_{tr}^{-1}(F(\gamma,s)\overline{\gamma(s^{-1}tr)}d\mu(\gamma))\\
&=\alpha_{r}(\Phi_{3}\circ \Phi_{2}\circ \Phi_{1}(F)(s,tr)). \end{aligned}$$
Hence $(\Phi_{3}\circ \Phi_{2}\circ \Phi_{1})\circ \hat{\hat{\alpha}}_{r}=((\mathrm{rt}\otimes\alpha)\otimes \mathrm{id})_{r}\circ (\Phi_{3}\circ \Phi_{2}\circ \Phi_{1})$.
\end{proof}


\subsection{The isometry property of $\Phi_4$}
$\\$
$\\$
The aim of this subsection is to prove the following.

\begin{prop}\label{equivariant for the compacts}
Let $\Phi_{4}:F^{p}(G,C_{0}(G,A),\mathrm{lt}\otimes \mathrm{id})\rightarrow \overline{M}_{G}^{p}\otimes_{p}A$ be defined in (\ref{Phi4}) under the assumption of Theorem \ref{T:main}. Then $\Phi_{4}$
is an isometric isomorphism which is equivariant for the action $(\mathrm{rt}\otimes \alpha)\otimes \mathrm{id}$ of $G$ on $F^{p}(G,C_{0}(G,A),\mathrm{lt}\otimes \mathrm{id})$ and the action $\mathrm{Ad}\rho\otimes\alpha$ of $G$ on $\overline{M}_{G}^{p}\otimes_{p} A$.
Moreover, the right-hand side equals $\mathcal{K}(l^{p}(G))\otimes_{p} A$ if and only if $p>1$ or $G$ is finite.
\end{prop}

To prove Proposition \ref{equivariant for the compacts}, we make some preparation.

\begin{lem}\label{C0}
Let $G$ be a countable discrete group and $A$ be a separable unital $L^p$ operator algebra.
If $A$ has unique $L^p$ operator matrix norms,
then so does $C_0(G,A)$.
\end{lem}

\begin{proof}
Let $\pi:C_0(G,A)\rightarrow \mathcal{B}(L^p(X,\mu))$ and $\sigma:C_0(G,A)\rightarrow \mathcal{B}(L^p(Y,\nu))$ be two isometric representations. By \cite[Lemma 5.15]{Lp AF}, $C_c(G,A)$ has unique $L^p$ operator norms. For any $F\in M_n(C_c(G,A))$, we have $\| \mathrm{id}_{M_{n}}\otimes\pi(F)\|=\| \mathrm{id}_{M_{n}}\otimes\sigma(F)\|$. Note that $C_c(G,A)$ is dense in $C_0(G,A)$.
Then $C_0(G,A)$ has unique $L^p$ operator matrix norms.
\end{proof}

The following result concerning compact operators on $L^{p}$ spaces will be useful in the proof of Proposition \ref{equivariant for the compacts}.

\begin{thm}[{\cite[Theorem 4.3]{Gardella and Thiel}}]\label{compact}
Let $p\in [1,\infty)$, $G$ be a discrete group, and $\mathrm{lt}:G\rightarrow \mathrm{Aut}(C_{0}(G))$ be the isometric action induced by the left translation of $G$ on itself. Then there are natural isometric isomorphisms:
$$F^{p}\left(G, C_{0}(G), \mathrm{lt}\right) \stackrel{\kappa}{\longrightarrow} F_{\lambda}^{p}\left(G, C_{0}(G), \mathrm{lt}\right) \stackrel{\iota}\longrightarrow \overline{M}_{G}^{p}.$$
Moreover, the right-hand side equals $\mathcal{K}\left(l^{p}(G)\right)$ when $p>1$, and is strictly smaller than $\mathcal{K}\left(l^{1}(G)\right)$ when $p=1$ and $G$ is infinite.
\end{thm}

\begin{proof}[Proof of Proposition \ref{equivariant for the compacts}.]
We give the proof by proving two key claims.

{\bf Claim 1.} $\Phi_4: F^{p}\left(G, C_{0}(G,A), \mathrm{lt}\otimes \mathrm{id}\right)\rightarrow \overline{M}_{G}^{p}\otimes_p A$ is an isometric isomorphism.

Since $G$ is amenable, by \cite[Theorem 7.1]{Phillips look like}), $F^{p}\left(G, C_{0}(G,A), \mathrm{lt}\otimes \mathrm{id}\right)$ is isometrically isomorphic to $F_{\lambda}^{p}\left(G, C_{0}(G,A), \mathrm{lt}\otimes \mathrm{id}\right)$; we denote by $\pi$ this isomorphism.
It suffices to show that $F^p_\lambda(G,C_0(G,A),\mathrm{lt}\otimes \mathrm{id})$ is isometrically isomorphic to $\overline{M}_{G}^{p}\otimes_p A$.

Firstly, we construct an isometric representation of $\overline{M}_{G}^{p}\otimes_p A$ and  $F^p_\lambda(G,C_0(G,A),\mathrm{lt}\otimes \mathrm{id})$ by using the property of unique $L^p$ operator matrix norms.
Since $A$ has unique $L^p$ operator matrix norms,
by Lemma \ref{isometric tensor}, the norm on spatial $L^p$ operator tensor product $\overline{M}_{G}^{p}\otimes_p A$ is independent of the choice of isometric representation of $A$.
Let $\pi_0:A\rightarrow \mathcal{B}(E)$ be a nondegenerate $\sigma$-finite isometric representation of $A$. Then there is a natural isometric representation of $\overline{M}_{G}^{p}\otimes_p A$ on $l^p(G,E)$.

Furthermore, there exists a  canonical nondegenerate $\sigma$-finite isometric representation $\pi_1$ of $C_0(G,A)$ on $l^p(G,E)$ which maps $f\otimes a$ to $M_f\otimes \pi_0(a)$, where $f\in C_0(G), a\in A$, and $M_f$ is the multiplication operator on $l^p(G)$. Let   $(\pi,\lambda_p^{l^p(G,E)})$ be its associated regular covariant representation. Then, by Lemmas \ref{does not depend} and \ref{C0}, the integrated form $\pi\rtimes\lambda_p^{l^p(G,E)}$ is an isometric representation of $F^p_\lambda(G,C_0(G,A),\mathrm{lt}\otimes \mathrm{id})$.


Now we are going to construct an isometric isomorphism from $F^p_\lambda(G,C_0(G,A),\mathrm{lt}\otimes \mathrm{id})$ to $\overline{M}_G^p\otimes_p A$.

We define an invertible isometry $V:l^p(G)\otimes_p l^p(G)\otimes_p E\rightarrow l^p(G)\otimes_p l^p(G)\otimes_p E$ as $\delta_s\otimes\delta_t\otimes\xi\mapsto \delta_{st}\otimes \delta_{t}\otimes \xi$, where $\delta_s$ is the canonical basis of $l^p(G)$ for all $s\in G$ and $\xi\in E$.
We denote by $I_{l^p(G)}$ and $I_E$ the identity operator on $l^p(G)$ and $E$, respectively. A direct computation shows that $$\begin{aligned} V\pi(f\otimes a)(\delta_s\otimes\delta_t\otimes\xi)&=
V(\delta_s\otimes\pi_0((\mathrm{lt}\otimes\mathrm{id})_{s^{-1}}(f\otimes a))(\delta_t\otimes\xi))\\
&= V(\delta_s\otimes M_{(\mathrm{lt})_{s^{-1}}(f)}\delta_t\otimes a\xi))\\
&=V(\delta_s\otimes f(st)\delta_t\otimes a\xi)\\
&=f(st)\delta_{st}\otimes \delta_{t}\otimes a\xi\\
&=(M_f\otimes I_{l^p(G)}\otimes a)(\delta_{st}\otimes \delta_{t}\otimes \xi)\\
&=(M_f\otimes I_{l^p(G)}\otimes a)(V(\delta_s\otimes \delta_t\otimes\xi)). \end{aligned}$$
It follows that \begin{align}\label{pi}
V\pi(f\otimes a)V^{-1}= M_f\otimes I_{l^p(G)}\otimes a.
\end{align}

It is easy to check that $\lambda_p^{l^p(G,E)}=\lambda_p\otimes I_{l^p(G)}\otimes I_E$, where $\lambda_p:G\rightarrow \mathcal{B}(l^p(G))$ is the left-regular representation of $G$ on $l^p(G)$.
For each $r\in G$, one can check that  $$\begin{aligned} V(\lambda_p(r)\otimes I_{l^p(G)} \otimes I_E)V^{-1}(\delta_s\otimes\delta_t\otimes\xi)&=
 V(\lambda_p(r)\otimes I_{l^p(G)} \otimes I_E)(\delta_{st^{-1}}\otimes\delta_{t}\otimes\xi)\\
&=V(\delta_{rst^{-1}}\otimes\delta_t\otimes \xi)\\
&=\delta_{rs}\otimes \delta_{t}\otimes\xi\\
&=(\lambda_p(r)\otimes I_{l^p(G)}\otimes I_E)(\delta_s\otimes\delta_t\otimes\xi). \end{aligned}$$
It follows that \begin{align}\label{lambda}
V( \lambda_p(r)\otimes I_{l^p(G)}\otimes I_E)V^{-1}=\lambda_p(r)\otimes I_{l^p(G)}\otimes I_E.\end{align}

Since $f\in C_0(G)$, by (\ref{pi}) and (\ref{lambda}), we have  $V(F^p_\lambda(G,C_0(G,A),\mathrm{lt}\otimes\mathrm{id}))V^{-1}\cong \overline{M}_G^p\otimes_p \mathbb{C}I_{l^p(G)}\otimes_p A$.
Obviously, $\overline{M}_G^p\otimes_p \mathbb{C}I_{l^p(G)}\otimes_p A$ is isometrically isomorphic to $\overline{M}_G^p\otimes_p A$.
Since $V$ is an isometry, it follows that $F^{p}_\lambda(G,C_{0}(G,A),\mathrm{lt}\otimes \mathrm{id})$ is isometrically isomorphic to $\overline{M}_G^p\otimes_p A$. We denote this isomorphism by $\varphi_4$.
It is easy to check that $\Phi_4=\varphi_4\circ \pi$. Hence $\Phi_4$ is an isometric isomorphism from $F^{p}(G,C_{0}(G,A),\mathrm{lt}\otimes \mathrm{id})$ onto $\overline{M}_G^p\otimes_p A$.
This proves Claim 1.

By Theorem \ref{compact}, $\overline{M}_G^p\otimes_p A$ equals $\mathcal{K}(l^{p}(G))\otimes_{p} A$ if and only if either $p>1$ or $G$ is finite.


{\bf Claim 2.} $\Phi_{4}$ is equivariant for the action $(\mathrm{rt}\otimes \alpha)\otimes \mathrm{id}$ of $G$ on $F^{p}(G,C_{0}(G,A),\mathrm{lt}\otimes \mathrm{id})$ and the action $\mathrm{Ad}\rho\otimes\alpha$ of $G$ on $\overline{M}_{G}^{p}\otimes_{p} A$.

For $F\in C_{c}(G\times G,A)\subseteq F^{p}(G,C_{0}(G,A),\mathrm{lt}\otimes \mathrm{id})$ and $h\in C_{c}(G,E)\subseteq l^{p}(G,E)$, we have $\Phi_{4}(F)h(t)=\sum_{s\in G}F(s,t)h(s^{-1}t)$. For $r\in G$, $T\in \mathcal{K}(l^p(G))$ and $a\in A$, we have $(\mathrm{Ad}\rho\otimes\alpha)_{r}(T\otimes a)=(\mathrm{Ad}\rho)_r(T)\otimes \alpha_r(a)=(\rho_r\otimes I_E)((I_{l^p{G)}}\otimes\alpha_r)(T\otimes a))(\rho_{r^{-1}}\otimes I_E).$
Hence
$$\begin{aligned} (\mathrm{Ad}\rho\otimes\alpha)_{r}(\Phi_{4}(F))h(t)&=
(\rho_r\otimes I_E)((I_{l^p{G)}}\otimes\alpha_r)(\Phi_{4}(F)))(\rho_{r^{-1}}\otimes I_E)h(t)\\
&=((I_{l^p{G)}}\otimes\alpha_r)(\Phi_{4}(F)))(\rho_{r^{-1}}\otimes I_E)h(tr)\\
&=\sum_{s\in G}\alpha_r(F(s,tr))(\rho_{r^{-1}}\otimes I_E)h(s^{-1}tr)\\
&=\sum_{s\in G}\alpha_r(F(s,tr))h(s^{-1}t)\\
&=\sum_{s\in G}((\mathrm{rt}\otimes\alpha)\otimes \mathrm{id})_{r}(F)(s,t)h(s^{-1}t)\\
&=\Phi_{4}( ((\mathrm{rt}\otimes\alpha)\otimes \mathrm{id})_{r}(F))h(t).
 \end{aligned}$$
Therefore, $(\mathrm{Ad}\rho\otimes\alpha)_{r}\circ\Phi_{4}=\Phi_{4}\circ( (\mathrm{rt}\otimes\alpha)\otimes \mathrm{id})_{r}$ as required.
\end{proof}

\begin{proof}[Proof of Theorem \ref{T:main}]
Statements (i), (ii) and (iii) in Theorem \ref{T:main} follow from Propositions \ref{change}, \ref{injective}, \ref{change for trival action} and \ref{equivariant for the compacts}. By Lemma \ref{1-3} and Proposition \ref{equivariant for the compacts},
the map $\Phi$ is equivariant for the
double dual action $\hat{\hat{\alpha}}$ of $G$ on $F^{p}(\hat{G},F^{p}(G,A,\alpha),\hat{\alpha})$, and the action $\mathrm{Ad}\rho\otimes\alpha$ of $G$ on $\mathcal{K}(l^{p}(G))\otimes_{p}A$.
\end{proof}


\section*{Acknowledgements}

The first author is supported by the National Natural Science Foundation of China (grant numbers 12201240, 11971253), the China Postdoctoral Science Foundation (grant number 2022M711310),  the Natural Science Foundation of Fujian Province (grant number 2020J05206) and the Scientific Research Project of Putian University (grant number 2020001).
The second author is supported by the National Natural Science Foundation of China (grant number 12171195).



\end{document}